\documentclass{amsart}
\usepackage{amssymb}
\usepackage{graphicx}
\usepackage{amscd}
\usepackage{color}

\allowdisplaybreaks
\numberwithin{figure}{section}
\numberwithin{table}{section}
\numberwithin{equation}{section}

\newtheorem{thm}{Theorem}[section]

\newtheorem{lem}[thm]{Lemma}
\newtheorem{cor}[thm]{Corollary}
\newtheorem{remark}{Remark}[section]
\newcommand{\C}{\mathbb{C}}
\newcommand{\N}{\mathbb{N}}
\newcommand{\R}{\mathbb{R}}

\newcommand{\B}{\mathcal B}

\newcommand{\cF}{\mathcal F}
\newcommand{\cG}{\mathcal G}
\newcommand{\cV}{\mathcal V}

\newcommand\interior{\operatorname{int}}
\newcommand{\Lc}{\mathcal L}

\newcommand\lip{\mathrm{lip}}
\newcommand\bg{\mathbf{F}}
\newcommand\g{F}

\newcommand\bL{\mathbf{L}}
\newcommand\bV{\mathbf{V}}

\newcommand\W{W}
\newcommand\amf{\mathfrak{a}}
\newcommand\bmf{\mathfrak{b}}
\newcommand\amh{\mathfrak{h}}
\newcommand\rmf{R}
\def\cprime{\char"7E }

\DeclareMathOperator{\Lip}{Lip}
\DeclareMathOperator{\diam}{diam}

\begin{document}

\title[High Order Computation of Hausdorff Dimension]
{Hidden Positivity and a New Approach to Numerical Computation of
Hausdorff Dimension: Higher Order Methods}
\author{Richard S. Falk}
\address{Department of Mathematics, Rutgers University, Piscataway, NJ 08854}
\email{falk@math.rutgers.edu}
\urladdr{http://www.math.rutgers.edu/\char'176falk/}

\author{Roger D. Nussbaum}
\address{Department of Mathematics,
Rutgers University, Piscataway, NJ 08854}
\email{nussbaum@math.rutgers.edu}
\urladdr{http://www.math.rutgers.edu/\char'176nussbaum/}
\subjclass[2000]{Primary 11K55, 37C30; Secondary: 65J10}
\keywords{Hausdorff dimension, positive transfer operators, 
continued fractions}
\date{original submitted August 24, 2020; revised version February 26, 2021}

\begin{abstract}
  In \cite{hdcomp1}, the authors developed a new approach to the
  computation of the Hausdorff dimension of the invariant set of an
  iterated function system or IFS. In this paper, we extend this
  approach to incorporate high order approximation methods. We again
  rely on the fact that we can associate to the IFS a parametrized
  family of positive, linear, Perron-Frobenius operators $L_s$, an
  idea known in varying degrees of generality for many years. Although
  $L_s$ is not compact in the setting we consider, it possesses a
  strictly positive $C^m$ eigenfunction $v_s$ with eigenvalue
  $\rmf(L_s)$ for arbitrary $m$ and all other points $z$ in the
  spectrum of $L_s$ satisfy $|z| \le b$ for some constant $b <
  \rmf(L_s)$. Under appropriate assumptions on the IFS, the Hausdorff
  dimension of the invariant set of the IFS is the value $s=s_*$ for
  which $\rmf(L_s) =1$.  This eigenvalue problem is then approximated
  by a collocation method at the extended Chebyshev points of each
  subinterval using continuous piecewise polynomials of arbitrary
  degree $r$.  Using an extension of the Perron theory of
  positive matrices to matrices that map a cone $K$ to its interior
  and explicit a priori bounds on the derivatives of the strictly
  positive eigenfunction $v_s$, we give rigorous upper and lower
  bounds for the Hausdorff dimension $s_*$, and these bounds converge
  rapidly to $s_*$ as the mesh size decreases and/or the polynomial
  degree increases.

\end{abstract}

\maketitle

\section{Introduction}
\label{sec:intro}

In this paper, we continue previous work in finding rigorous estimates
for the Hausdorff dimension of invariant sets for iterated function
systems or IFS's. To describe the framework of the problem we are
considering, we let $S \subset \R$ be a nonempty compact set, and for
some positive integer $m$, let $\theta_p: S \to S$ and $g_p : S \to
[0, \infty] \in C^m(S)$ for $1 \le p \le n < \infty$.  If $\theta_p$
are contraction mappings, it is known that there exists a unique,
compact, nonempty set $C \subset S$ such that $C= \cup_{p=1}^n
\theta_p(C)$.  The set $C$ is called the invariant set for the IFS
$\{\theta_p: 1 \le p \le n\}$.

For $s >0$, define a bounded linear map $L_s: C(S) \to C(S)$, (often
called a {\it Perron-Frobenius operator} or {\it linear transfer operator}) by
\begin{equation}
\label{2.1}
(L_s f)(t) = \sum_{p=1}^n [g_p(t)]^s f(\theta_p(t)), \quad t \in S.
\end{equation}
Under additional appropriate hypotheses (stated in the next section),
$L_s$, considered as a map from $C^m(S) \mapsto C^m(S)$, has a
strictly positive eigenfunction $v_s \in C^m(S)$ with algebraically
simple eigenvalue $\lambda_s = \rmf(L_s)$, the spectral radius of
$L_s$. In addition, all other points $z$ in the spectrum of $L_s$
satisfy $|z| \le b$ for some constant $b < \rmf(L_s)$.  A more precise
statement of this result, along with other conclusions, is given in
Theorem~\ref{thm:2.4} in the next section. Note that in the $C^m$
setting, $L_s$ is, in general, not compact, has positive essential
spectral radius and cannot be the limit in operator norm of a sequence
of finite dimensional linear operators.  These difficulties do not
usually arise if $L_s$ can be studied in a Banach space of complex
analytic functions; and there is an extensive literature concerning
the spectral theory of Perron-Frobenius operators which map a Banach
space of analytic functions into itself.  We prefer to work in the
more general $C^m$ setting so as to provide tools which also can be
applied to some non-analytic examples, e.g., as in Section 5 of
\cite{hdcomp1}.

The aim of this paper is to derive an approximation scheme that allows
us to estimate $\rmf(L_s)$ by the spectral radius of an associated
matrix $\bL_s$ which approximates the operator $L_s$ in a weak sense
and then to obtain rigorous bounds on the error $|\rmf(L_s) -
\rmf(\bL_s)|$. We then use this approximation scheme to estimate
$s_*$, the unique number $s \ge 0$ such that $\rmf(L_s) =1$.  Under
appropriate assumptions, $s_*$ equals the Hausdorff dimension of the
invariant set associated to the IFS. This observation about Hausdorff
dimension has been made, in varying degrees of generality by many
authors. See, for example, \cite{Bumby1}, \cite{Bumby2}, \cite{Bowen},
\cite{Cusick1}, \cite{Cusick2}, \cite{Falconer}, \cite{Good},
\cite{Hensley1}, \cite{Hensley2}, \cite{Hensley3}, \cite{J},
\cite{Jenkinson}, \cite{Jenkinson-Pollicott}, \cite{MR1902887},
\cite{H}, \cite{Mauldin-Urbanski}, \cite{N-P-L}, \cite{Ruelle},
\cite{Ruelle2}, \cite{Rugh}, \cite{Schief}, and \cite{C-L-U}.  There
is also a large literature on the approximation of linear transform
operators, not necessarily related to the computation of Hausdorff
dimension, and often assuming the maps are analytic.  We do not
attempt to survey that literature, other than to cite one recent
paper, \cite{B-S}, which has some connections to our work here, and
contains many references to that literature.

In previous work, \cite{hdcomp1}, the authors presented a new approach
to the problem described in the preceding paragraph. We obtained
rigorous upper and lower bounds for the Hausdorff dimension $s_*$, and
these bounds exhibited second order convergence to $s_*$ as the mesh
size decreases.  The approximate matrix was obtained by a collocation
method using continuous piecewise linear functions, motivated by the
fact that if such functions are nonnegative at the mesh points, they
are nonnegative at all points of the interval in which they are
defined.  This property leads to nonnegative matrix approximations of
the operator $L_s$.  One would like these matrices to mimic the
properties of the continuous operator $L_s$, which means they should
satisfy the conclusions of the Perron theorem for positive matrices
(matrices with strictly positive entries), i.e., they should have an
eigenvalue of multiplicity one equal to the spectral radius of the
matrix with corresponding positive eigenvector and all other
eigenvalues of the matrix should have modulus less than the spectral
radius. This is not true for nonnegative matrices, however, unless
they have an additional property.  One such property that would
guarantee this is that the matrix $\bL_s$ be {\it primitive}, i.e.,
there exists a positive integer $p$ such that $\bL_s^p$ is a positive
matrix.  Note that if $\bL_s$ is {\it irreducible}, then the first two
properties hold, but there can be other eigenvalues of the same
modulus as the spectral radius. Unfortunately, the approximation
scheme used led to matrices which are neither primitive nor
irreducible..  The remedy to obtain the desired properties was to note
that the cone $K$ of nonnegative vectors is not the natural cone in
which such matrices should be studied. Using a more general notion of
positivity of an operator $L$ in which $L$ maps a cone $K$ into
itself, one can still obtain the conclusions of the Perron theorem.
This is important since we use the spectral radius of the approximate
matrix $\bL_s$ to approximate the spectral radius of $L_s$ and the
fact that there is a single dominant eigenvalue enables us to
calculate it efficiently using some variant of the power method.

In this paper, we analyze a similar method obtained by approximation
using higher order piecewise polynomials. As we shall see, the
matrices resulting from the approximation scheme appear to be even
more problematic, since they are not even nonnegative.  Despite this
fact, the use of an alternative cone, in place of the standard cone of
nonnegative vectors, allows us to show that the conclusions of the
classical Perron theorem also hold for the matrices of this paper.
There is a substantial abstract theory which has been developed for
finite dimensional linear operators which are positive in the sense
that they map a cone into itself. The survey paper \cite{Tam},
references in \cite{Tam}, and appendices A and B in
\cite{Lemmens-Nussbaum-2013} provide a good starting point. However,
the difficulty lies in finding such a cone that fits the application
under study. We use the term {\it hidden positivity} to call attention
to the fact that we are able to find such a cone for the approximate
operators developed in this paper.

The cone we use is easiest to describe in the case of continuous, piecewise
linear functions, and is defined as follows.  On the interval $[0,1]$,
for fixed integer $N$, let $h = 1/N$ and $x_i = i h$, $i = 0,1, \ldots, N$.
The space of continuous, piecewise linear functions is just the finite
dimensional space of continuous functions that restricted to each subinterval
$[x_i, x_{i+1}]$ are linear functions.  Since a function $w$ in this space
is completely determined by its values $w_i = w(x_i)$, $i = 0,1, \ldots, N$
(the {\it degrees of freedom} of $w$), we can also view $w$ as the
vector $[w_0, \ldots, w_N]$.  For any integer $M >0$, we then define
the cone $K_M$ by
\begin{equation*}
K_M = \{w: w_i \le \exp(M |x_i- x_j|) w_j,  \quad i,j =0, 1, \ldots N\}.
\end{equation*}
The cone for higher order piecewise polynomials is similar, but its
description is more involved because of the more complicated nature of
the degrees of freedom of such functions.  The details are provided in
Section~\ref{sec:approx}.

One technical difference between piecewise linear functions and higher
order piecewise polynomials is that in order to obtain the results
described above, we must consider approximations $\bL_{s, \nu}$ of the matrix
$L_s^{\nu}$, where $\nu$ depends on the degree $r$ of the piecewise
polynomial approximation. As we observe in the next section, the operator
$L_s^{\nu}$ has the same form as $L_s$, i.e.,
\begin{equation*}
(L_s^{\nu} f)(t) = \sum_{\omega \in \Omega_{\nu}} [g_{\omega}(t)]^s f(\theta_{\omega}(t)),
\end{equation*}
where for $\nu \ge 1$,
$\Omega_{\nu} = \{\omega = (p_1, p_2, \ldots, p_{\nu}): 1 \le p_j \le n
\text{ for } 1 \le j \le \nu\}$; and for
$\omega = (p_1, p_2, \ldots, p_{\nu}) \in \Omega_{\nu}$,
\begin{equation*}
\theta_{\omega}(t) = (\theta_{p_1} \circ \theta_{p_2} \circ
\cdots \circ \theta_{p_{\nu}})(t),
\end{equation*}
and $g_{\omega}(t)$ is defined in the next section.  We note that
under the weaker assumption that $\theta_{\omega}$ is a contraction
mapping for all $\omega \in \Omega_{\nu}$, there exists a unique
compact set $C$ such that $C= \cup_{\omega \in
  \Omega_{\nu}}\theta_{\omega}(C)$ and that necessarily $C =
\cup_{p=1}^n \theta_p(C)$.  By using the matrix $\bL_{s, \nu}$, one
reduces the domain of the operator to a finite set of subintervals
whose total length is much less than the original length of the
domain, resulting in many fewer mesh points.  The downside, however,
is that this advantage is completely offset by the increase in the
number of terms in the operator $L_s$ for each time the
map is iterated, i.e., from $n$ to $n^{\nu}$.  We note that since
we have not found any case where the method fails if we do not iterate
the matrix, we conjecture that this extra condition is an artifact of
the method of proof, and not the method itself.

To obtain the conclusions of the Perron theorem, the key
result is to show that for some $0 < M' < M$,
\begin{equation}
\label{keycond-intro}
\bL_{s,\nu}(K_M \setminus \{0\}) \subset K_{M'} \setminus \{0\}.
\end{equation}
This enables us to apply
results from the literature on mappings of a cone to itself to obtain
the desired conclusions.  Details of this connection, along with
references to the relevant literature, are described in
Section~\ref{sec:theory-c}. 

A main goal of our approach, in addition to proposing a new
approximation scheme, is to provide rigorous upper and lower bounds
for the Hausdorff dimension of the underlying IFS.  This will follow
directly if we are able to derive rigorous error bounds for
$|[\rmf(L_s)]^{\nu} - \rmf(\bL_{s,\nu})|$.  In the case of piecewise
linear functions, we obtained the bounds by using a simple and
well-known result (c.f. Lemma 2.2 of \cite{hdcomp1}) that if $A$ is an
nonnegative matrix and $w$ a vector with strictly positive components,
then if for all components $k$, (i) if $(A w)_k \ge \lambda w_k$, then
$\rmf(A) \ge \lambda$ and (ii) if $(A w)_k \le \lambda w_k$, then
$\rmf(A) \le \lambda$.  Here, we use an analogous result for a matrix
mapping a cone $K$ to itself, in which we replace $\le$ by $\le_K$, 
i.e, $u \le_K v$ if and only if $v-u \in K$.

Another key tool for obtaining rigorous upper and lower bounds for the
Hausdorff dimension $s_*$, is to obtain and use explicit a priori
bounds on the quantity $D^q v_s(x)/v_s(x)$ of the strictly positive
eigenfunction $v_s$ of $L_s$, where $D^q v_s$ denotes the $q$-th derivative of
$v_s$. Such estimates are derived in Section~\ref{sec:est-constants}.

In order to improve the efficiency of our computation, we consider in
Section~\ref{sec:est-constants} the possibility of replacing the
original interval $S = [a,b]$ by a smaller interval $S_0 \subset S$
such that $\theta_{\beta}(S_0) \subset S_0$ for $\beta \in \B$.  In
particular, for the maps $\theta_{\beta} =1/(x+ \beta)$, setting
$\gamma= \min \{\beta: \in \B\}$ and $\Gamma = \max \{\beta: \in
\B\}$, we can reduce the interval $S$ to $[\amf_{\infty},
\bmf_{\infty}]$, where
\begin{equation*}
\amf_{\infty} = - \frac{\gamma}{2} + \sqrt{(\gamma/2)^2 + (\gamma/\Gamma)}
\quad \text{and} \quad
\bmf_{\infty} = - \frac{\Gamma}{2} + \sqrt{(\Gamma/2)^2 + (\Gamma/\gamma)}
= \frac{\Gamma}{\gamma} \amf_{\infty}.
\end{equation*}
For example, for the set $\{1,2\}$, we reduce the interval $[0,1]$
to $[(\sqrt{3} - 1)/2, \sqrt{3} - 1]$ of length $0.366$, while for
the set  $\{10,11\}$, we reduce the interval $[0,1]$
to $[0.0901, .0991]$ of length $.009$.

A main result of the paper (Theorem~\ref{thm:8.1}) says that under appropriate
hypotheses, there is a computable constant $H$, such that
\begin{equation*}
[\rmf([1 + H h^r]^{-1} \bL_{s,\nu})]^{1/\nu} \le \lambda_s \le
[\rmf([1 - H h^r]^{-1} \bL_{s,\nu})]^{1/\nu},
\end{equation*}
where $h$ denotes the maximum mesh size and $r$ the degree of the piecewise
polynomial approximation. 
Using these inequalities, we can obtain rigorous upper and lower bounds
on the Hausdorff dimension of the invariant set associated with the transfer
operator $L_s$ as follows.
Let $s_l$ and $s_u$ denote values of $s$ satisfying
\begin{equation*}
[1-Hh^r]^{-1} \rmf(\bL_{s_u,\nu}) <1, \qquad
[1+Hh^r]^{-1} \rmf(\bL_{s_l,\nu}) >1.
\end{equation*}
It follows immediately from Theorem~\ref{thm:8.1} that
$\lambda_{s_u}^{\nu} < 1$ and $\lambda_{s_l}^{\nu} >1$.  Since the
spectral radius $\lambda_s$ of $L_s$ is a strictly decreasing function
of $s$, there will be a value $s_*$ satisfying $s_l < s_* < s_u$ for
which $\lambda_{s_*}^{\nu} = 1$, or equivalently $\lambda_{s_*} = 1$.
The value $s_*$ gives the Hausdorff dimension $s_*$ of the invariant
set associated with the transfer operator $L_s$.  Since $s_u-s_l$ is
of order $h^r$, by choosing $h$ to be sufficiently small and/or $r$ to
be sufficiently large, we obtain a highly accurate estimate for $s_*$.
As noted above, for a given $s$, $\rmf([1 \pm H h^r]^{-1}
\bL_{s,\nu})$ is easily computed by variants of the power method for
eigenvalues, since the largest eigenvalue has multiplicity one and is
the only eigenvalue of its modulus. Our theoretical results imply that
$\bL_{s,\nu}$ has an eigenvector $w$ in $K:=K_M$ with eigenvalue
$\rmf(L_{s,\nu})$ and that this eigenvector can be computed to high
accuracy. Still, one might be concerned about possible errors in the
computation of of $R(\bL_{s,\nu})$ and $w$. However, independently of
how a purported eigenvector $w \in K$ for $\bL_{s,\nu}$ is found, if
$\alpha w \le_K \bL_{s,\nu} w \le_K \beta w$,
Lemma~\ref{lem:cone-compare} in Section~\ref{sec:theory-c} implies
that $\alpha \le \rmf(\bL_{s,\nu}) \le \beta$. This provides a means
of giving rigorous bounds for $\rmf(\bL_{s,\nu})$.

In Section~\ref{sec:num-comp}, we present results of computations of
the Hausdorff dimension $s$ of invariant sets in $[0,1]$ arising from
continued fraction expansions. In this much studied case, one defines
$\theta_p = 1/(x+p)$, for $p$ a positive integer and $x \in [0,1]$;
and for a subset $\B \subset \N$, one considers the IFS $\{\theta_p: p
\in \B\}$ and seeks estimates on the Hausdorff dimension of the
invariant set $C =C(\B)$ for this IFS.  This problem has previously
been considered by many authors. See \cite{Bourgain-Kontorovich},
\cite{Bumby1}, \cite{Bumby2}, \cite{Good}, \cite{Hensley1},
\cite{Hensley2}, \cite{Hensley3}, \cite{Jenkinson},
\cite{Jenkinson-Pollicott}, and \cite{Heinemann-Urbanski}.  In this
case, \eqref{2.1} becomes
\begin{equation*}
(L_{s} f)(x) = \sum_{p \in \B} \Big(\frac{1}{x+p}\Big)^{2s} 
f\Big(\frac{1}{x+p}\Big), \qquad 0 \le x \le 1,
\end{equation*}
and one seeks a value $s \ge 0$ for which $\lambda_s:= \rmf(L_{s})
=1$.  Several of the papers listed above contain a large number
of computations to various degrees of accuracy of the Hausdorff
dimension of the IFS $\{\theta_p: p \in \B\}$, for various
choices of the set $\B$.  An early paper, \cite{Hensley2}, gives
results for over 30 choices of $\B$, containing between two and five
terms in the set $\B$, with results reported to an accuracy between
$10^{-6}$ and $10^{-19}$, depending on the problem studied.  A {\it
  Mathematica} code implementing the algorithm is also provided.  In
\cite{Jenkinson}, computations to four decimal places are given for
over $35$ choices of the set $\B$, ranging from two terms, to as many
as 34 (this computation is to three decimal places), and also includes
a computation of $E[1,2]$, ($\B = \{1,2\}$), accurate to 54 decimal
places.  In \cite{Jenkinson-Pollicott}, eight examples of $\B$,
consisting of two terms, are computed with accuracies ranging from
$10^{-13}$ to $10^{-52}$, depending on the choice of $\B$, although
the authors note that for the sets $[10,11]$, and $[100, 10,000]$,
they were able to compute to accuracies of $10^{-61}$ and $10^{-122}$,
respectively.  This depends on the fact that the speed of convergence
of their methods depends on the size of the smallest value of $p \in
\B$.  In \cite{JP100}, the Hausdorff dimension of $E[1,2]$ is
rigorously computed to 100 decimal places, although more digits are
computed.  It is less clear how well some of the approximation schemes
employed in these papers work when $|\B|$ is moderately large or when
different real analytic functions $\hat \theta_j: [0,1] \to [0,1]$ are
used.  Here and in \cite{hdcomp1}, in the one dimensional case, we
present an alternative approach with much wider applicability that
only requires the maps in the IFS to be $C^m$, for some finite value
of $m$.  As an illustration, we considered in \cite{hdcomp1},
perturbations of the IFS for the middle thirds Cantor set for which
the corresponding contraction maps are $C^3$, but not $C^4$.

The computations in Section~\ref{sec:num-comp} include choices of
various sets of continued fractions, maximum mesh size $h$, piecewise
polynomial degree $r$, and number of iterations $\nu$, ( where $\nu=1$
corresponds to the original map), including choices of $\nu$ for which
the hypotheses of our theorems are satisfied, but also computations
which obtain the same results when the mappings are not
iterated. These results support our conjecture that our method also
works in the non-iterated situation. To facilitate computation of
further examples, a {\it Matlab} code is provided in {\tt
  https://sites.math.rutgers.edu/\char'176falk/hausdorff/codes.html}.

An outline of the paper is as follows. In the next section, we
introduce further notation and state some preliminary results we will
use in our analysis. Section~\ref{sec:approx} contains a description
of the approximate problem and the cone we use to analyze
it. Section~\ref{sec:theory-c} contains the theoretical results we
will need to show that the matrices arising from the approximation
scheme satisfy the conclusions of the Perron theorem.  In
Section~\ref{sec:theory-d}, the main result is to determine conditions
under which the matrix $\bL_{s,\nu}$ satisfies \eqref{keycond-intro}.
These conditions involve a number of constants, which we then estimate
in Section~\ref{sec:estimating-sr}, ultimately deriving bounds for
$\rmf(L_s)$ in terms of $[\rmf(\bL_{s,\nu})]^{1/\nu}$.  In
Section~\ref{sec:est-constants}, we consider a method for reducing the
size of the interval $S$ on which the problem is defined, with the aim
of reducing the number of mesh points that will be needed in the
approximation scheme.  In so doing, we are also able to improve the
bound on two constants which are used in the error estimate for
$\rmf(L_s)$.  Recall that condition \eqref{keycond-intro} requires
determining for each constant $M$, a constant $0 < M' <M$ such that
\eqref{keycond-intro} is satisfied.  In Section~\ref{sec:computation},
we provide a procedure for determining this constant.  Finally, the
numerical computations described above are given in
Section~\ref{sec:num-comp}.

It would be of considerable interest to extend the methods of this
paper to the two dimensional case, e.g., to the problem of obtaining
rigorous estimates for the Hausdorff dimension of sets of complex
continued fractions. We conjecture that such an extension can be done,
but we leave it as an open problem for possible future work.

\section{Notation and Preliminaries}
\label{sec:notation}

Let $C(S)$ denote the Banach space of continuous functions $f :S \to \R$, where
$S$ is a compact subset of $\R$. Assume

(H0): For $1 \le p \le n < \infty$, $\theta_p:S \to S$ is a
Lipschitz map.

(H1): For $1 \le p \le n < \infty$, $g_p:S \to [0, \infty)$ 
is a nonnegative continuous function which is not identically zero.
In addition, there exists a constant $M_0 >0$ such that
\begin{equation*}
g_p(t_1) \le g_p(t_2) \exp(M_0|t_1 - t_2|), \quad \forall t_1, t_2 \in S, 
\quad 1 \le p \le n.
\end{equation*}
We note that it is easy to show that (H1) is equivalent to assuming
that $g_p(t) >0$ for all $t \in S$, and
\begin{equation*}
|\ln(g_p(t_1)) - \ln(g_p(t_2))| \le M_0 |t_1-t_2|, \qquad \forall t_1,t_2
 \in S, \quad 1 \le p \le n.
\end{equation*}

For $s >0$, define a bounded linear map $L_s:C(S) \to C(S)$ (often called a
{\it Perron-Frobenius operator}) by \eqref{2.1}, i.e.
\begin{equation*}
(L_sf)(t) = \sum_{p=1}^n [g_p(t)]^s f(\theta_p(t)), \quad t \in S.
\end{equation*}

We shall need to consider the $\nu$th iterate of $L_s$, $L_s^{\nu}$.  
For $\nu \ge 1$,
let $\Omega_{\nu} = \{\omega = (p_1, p_2, \ldots, p_{\nu}): 1 \le p_j \le n
\text{ for } 1 \le j \le \nu\}$; and for
$\omega = (p_1, p_2, \ldots, p_{\nu}) \in \Omega_{\nu}$, define
\begin{equation*}
\theta_{\omega}(t) = (\theta_{p_1} \circ \theta_{p_2} \circ
\cdots \circ \theta_{p_{\nu}})(t)
\end{equation*}
and
\begin{equation*}
g_{\omega}(t) = g_{p_1}(\theta_{p_2} \circ \cdots \circ \theta_{p_{\nu}}(t))
g_{p_2}(\theta_{p_3} \circ \cdots \circ \theta_{p_{\nu}}(t))
\cdots g_{p_{\nu-1}}(\theta_{p_{\nu}}(t)) g_{p_{\nu}}(t).
\end{equation*}
The reader can verify (e.g., see \cite{N-P-L}) that for all $f \in C(S)$,
\begin{equation*}
(L_s^{\nu}f)(t) = \sum_{\omega \in \Omega_{\nu}} [g_{\omega}(t)]^s f(\theta_{\omega}(t)).
\end{equation*}
Note that $L_s^{\nu}$ has the same form as $L_s$, except with index
set $\Omega_{\nu}$.  To analyze the operator $L_s^{\nu}$, we shall need
stronger assumptions than (H0).  We will thus assume

(H2): (H0) is satisfied and there exist constants $C_0 \ge 1$ and 
$\kappa$, $0 < \kappa <1$, such that for all integers $\nu \ge 1$, all $\omega
\in \Omega_{\nu}$, and all $t_1, t_2 \in S$,
\begin{equation*}
|\theta_{\omega}(t_1) - \theta_{\omega}(t_2)| \le C_0 \kappa^{\nu} |t_1 - t_2|.
\end{equation*}

Assuming (H1) and (H2), one can prove that for all $\omega \in \Omega_{\nu}$ and
all $t_1, t_2 \in S$, 
\begin{equation*}
g_{\omega}(t_1) \le \exp(M_0'|t_1-t_2|) g_{\omega}(t_2),
\end{equation*} 
where $M_0' = M_0 C_0[(1- \kappa^{\nu})/(1- \kappa)]$.
The proof is left to the reader. The reader will notice that the above
framework carries over to the more general case that $S$ is a compact
metric space with metric $\rho$. (H1) and (H2) take the same form except
that $|t_1-t_2|$ is replaced by $\rho(t_1,t_2)$.

The following result provides some theoretical background which will be
essential for our later work concerning the operator $L_s$. This
theorem is a special case of Corollary 6.6 in \cite{Nussbaum-2016}. We
refer to Section 3 of \cite{Nussbaum-1970} for a brief discussion of the
essential spectrum, which is mentioned in Theorem~\ref{thm:2.4} below.

\begin{thm}
\label{thm:2.4}
Assume hypotheses (H1) and (H2) are satisfied, that $S$ is a finite
union of compact intervals, and that $L_s$ is given by \eqref{2.1},
where $s >0$.  Assume also that $\theta_i \in C^m(S)$ and $g_i \in
C^m(S)$ for some positive integer $m$. Let $\Lambda_s:Y:=
C^m(S) \to Y$ be the bounded linear operator given by
\eqref{2.1}, but considered as a map of $Y \mapsto Y$, so $L_s(f) =
\Lambda_s(f)$ for $f \in Y$. If $\rmf(L_s)$ (respectively,
$\rmf(\Lambda_s)$) denotes the spectral radius of $L_s$ (respectively, of
$\Lambda_s$) and $\rho(\Lambda_s)$ denotes the essential spectral
radius of $\Lambda_s$ and $\kappa$ is as in (H2), then
\begin{equation*}
\rho(\Lambda_s) \le \kappa^m \rmf(\Lambda_s), \qquad
\rmf(\Lambda_s) =  \rmf(L_s):= \lambda_s >0.
\end{equation*}
Let $\hat \Lambda_s$ denote the complexification of $\Lambda_s$. If
$\sigma(\hat \Lambda_s)$ denotes the spectrum of $\hat \Lambda_s$, and we
define $\sigma(\Lambda_s) := \sigma(\hat \Lambda_s)$, then if $z \in 
\sigma(\Lambda_s)$ and $\rho(\Lambda_s) < |z|$, $z$ is an isolated point
of $\sigma(\Lambda_s)$ and is an eigenvalue of $\Lambda_s$ of finite
algebraic multiplicity.  Also, there exists $b_s < \lambda_s$ such that
\begin{equation*}
\sigma(\Lambda_s) \setminus \{\lambda_s\} \subset \{z \in \C : |z| \le |b_s|\}.
\end{equation*}
There exists a strictly positive eigenfunction $v_s \in C^m(S)$ with
eigenvalue $\lambda_s >0$, and $\lambda_s$ is an algebraically simple eigenvalue
of $\Lambda_s$.  If $u \in Y$ and $u(t) >0$ for all $t \in S$,
there exists a positive real number $\alpha$ (dependent on $u$) such that
\begin{equation}
\label{2.14}
\lim_{k \rightarrow \infty} \Big(\frac{1}{\lambda_s}\Big)^k \Lambda_s^k (u) = \alpha
v_s,
\end{equation}
where the convergence is in the norm topology on $Y$.
\end{thm}

\begin{remark}
\label{rem:2.5}
In our work here, it will be important to have estimates on
\begin{equation*}
\sup \Big\{\frac{|(d^j v_s/dt^j)(t)|}{v_s(t)}: t \in S\Big\},
\end{equation*}
where $1 \le j \le m$.  Note that if we take $u:=1$ in 
\eqref{2.14}, we find that for $t \in S$ and $1 \le j \le m$,
\begin{equation}
\label{2.15}
\frac{|(d^jv_s/dt^j)(t)|}{v_s(t)} = \lim_{k \rightarrow \infty}
\frac{|\sum_{\omega \in \Omega_k} (d^j g_{\omega}^s/dt^j)(t)|}
{\sum_{\omega \in \Omega_k} g_{\omega}(t)^s},
\end{equation}
and the convergence in \eqref{2.15} is uniform in $t \in S$.
If $u$ is as in \eqref{2.14}, we also obtain from \eqref{2.14} that
\begin{equation*}
\lim_{k \rightarrow \infty} \frac{\Lambda_s^{k+1}u}{\Lambda_s^k u} = \lambda_s,
\end{equation*}
where the convergence to the constant function $\lambda_s$ is in the norm
topology on $Y$.
\end{remark}


\section{Approximation of the spectral radius of $L_s$}
\label{sec:approx}
Returning to the notation of \eqref{2.1}, we want to approximate $\rmf(L_s)$ by
the spectral radius of an appropriate finite dimensional linear map
$\bL_s$. To do so, we assume that $S = [a,b]$ in (H1) and (H2), with $a <b $
and let $\hat S$ denote a union of
disjoint subintervals $[a_i, b_i] \subset [a,b]$, $i=1, \ldots, I$.
We also assume throughout this section that $\theta_p(\hat S) \subset \hat S$
for $1 \le p \le n$.
Further subdivide each interval $[a_i, b_i]$ into $N_i$ equally spaced
subintervals $[t_{j-1}^i,t_j^i]$, $j =1, \ldots, N_i$ of width
\begin{equation}
\label{2.30}
h_i = (b_i-a_i)/N_i, \qquad 1 \le i \le I.
\end{equation}
Set $h = \max_{1 \le i \le I} h_i$.

Next let $\{c_{j,k}^i\}_{k=0}^r \in [t_{j-1}^i, t_j^i]$, with $c_{j,0}^i =
t_{j-1}^i$, $c_{j,r}^i = t_j^i$, and $c_{j,k}^i < c_{j,k+1}^i$ for $0 \le k <
r$.  Given values $\{F_{j,k}^i\} = F(c_{j,k}^i)$, we then define on
$\hat S$, a piecewise polynomial $\cF$ as follows: For $t_{j-1}^i \le x \le
t_{j}^i$, $1 \le j \le N_i$, and $1 \le i \le I$,
\begin{equation}
\label{2.31}
\cF|_{[t_{j-1}^i, t_j^i]}(x) = \cF_j^i(x) = \sum_{k=0}^r l_{j,k}^i(x) F_{j,k}^i,
\end{equation}
where
\begin{equation*}
l_{j,k}^i(x) = \frac{\prod_{\substack{l=0 \\ l \neq k}}^r (x-c_{j,l}^i)}
{\prod_{\substack{l=0 \\ l \neq k}}^r(c_{j,k}^i -c_{j,l}^i)}.
\end{equation*}
Since $c_{j,r}^i = c_{j+1,0}^i$, $\cF \in \bV_h^r$, the space of
continuous piecewise polynomials of degree $\le r$, 
whose degrees of freedom are the $Nr+I:=Q$ values $F_{j,k}^i$,
where $N = \sum_{i=1}^I N_i$.

We note that we can simplify our expressions by choosing points
$\{\hat c_k\}_{k=0}^r \in [-1,1]$, with $\hat c_k < \hat c_{k+1}$ for
$0 \le k < r$, $\hat c_0 = -1$ and $\hat c_r =1$. If we then define
\begin{equation*}
c_{j,k}^i = t_{j-1}^i + h_i(1 + \hat c_k)/2,
\end{equation*}
and write $x \in [t_{j-1}^i,t_j^i] \subset [a_i,b_i]$ in the form $x = t_{j-1}^i +
h_i(1 + \hat x)/2$, where $\hat x \in [-1,1]$, we obtain
\begin{equation}
\label {2.32}
l_{j,k}^i(x) = \hat l_{k}(\hat x)
= \frac{\prod_{\substack{l=0 \\ l \neq k}}^r (\hat x- \hat c_l)}
{\prod_{\substack{l=0 \\ l \neq k}}^r(\hat c_{k} - \hat c_{l})}.
\end{equation}

Because we seek to make use of high order piecewise polynomials, it is
important to choose the points $\hat c_k$ to avoid the large errors
that can occur in polynomial interpolation due to Runge's phenomenon
(e.g., when equally spaced interpolation points are used).  Since, for
our analysis, we shall need the function $\cF(x)$ in \eqref{2.31} to be
continuous, we choose the points $\hat c_k$ to be the extended
Chebyshev points in $[-1,1]$ given by
\begin{equation}
\label{hatck}
\hat c_k = - \cos \Big(\frac{2k+1}{2r+2} \pi\Big)\Big/
\cos\Big(\frac{\pi}{2r+2}\Big), \quad k=0, \ldots, r,
\end{equation}
obtained by rescaling the usual Chebyshev nodes.  Then
\begin{equation}
\label{2.33}
c_{j,k}^i = t_{j-1}^i + \frac{h_i}{2} 
\Big(1 - \Big[\cos \Big(\frac{2k+1}{2r+2}\pi\Big)\Big/
\cos \Big(\frac{\pi}{2r+2}\Big)\Big]\Big), \quad k=0, \ldots, r.
\end{equation}
We note that another possible choice is to use the augmented Chebyshev points,
consisting of the roots of the Chebyshev polynomial of degree $r-1$ shifted
to the interval $[t_{j-1}^i, t_j^i]$ plus the endpoints $t_{j-1}^i$ and $t_j^i$.

With this notation, we can now define, for $s >0$, the linear map
$\bL_s: \R^Q \to \R^Q$. If $\bg = \{F_{j,k}^i\} \in \R^Q$, we define
\begin{equation*}
\bL_s(\bg)(c_{j,k}^i) = \sum_{p=1}^n g_p(c_{j,k}^i)^s \cF(\theta_p(c_{j,k}^i)),
\end{equation*}
where $\cF$ is defined above.
Equivalently, we can also think of the operator $\bL_s$ as a map from the space
$\bV_h^r \to \bV_h^r$, if we replace $\bg$ by $\cF$, and given the values
$\{G_{j,k}^i\} = \sum_{p=1}^n g_p(c_{j,k}^i)^s \cF(\theta_p(c_{j,k}^i))$,
we define $\cG(x)$ as follows: For $t_{j-1}^i \le x \le t_j^i$, $1 \le j \le
N_i$, and $1 \le i \le I$,
\begin{equation*}
\cG|_{[t_{j-1}^i,t_j^i]}(x) = \cG_j^i(x) = \sum_{k=0}^r l_{j,k}^i(x) G_{j,k}^i.
\end{equation*}

Given a positive real number $M$, we next define $K_M \subset \R^{Q}$
as $\{\bg \in \R^{Q}\}$ such that for all $\xi = c_{j,k}^i$ and
$\eta = c_{j',k'}^{i'}$, with $1 \le i,i' \le I$, $1 \le j \le N_i$,
$1 \le j' \le N_{i'}$, and $0 \le k,k' \le r$,
\begin{equation}
\label{2.35}
F(\xi) \le \exp(M|\xi- \eta|) F(\eta). 
\end{equation}
Note that to verify \eqref{2.35}, it suffices to verify it whenever 
$\xi$ and $\eta$ are two consecutive points in the linear ordering
inherited from $\R$ of the points $\{ c_{j,k}^i\}$.

One can easily verify that if $\bg \in K_M$, then either (a) $F_{j,k}^i =0$
for all $1 \le i \le I$, $1 \le j \le N_i$, and $0 \le k \le r$, or (b)
$F_{j,k}^i >0$ for all $i,j,k$ in this range. In case (b), one has for
all $1 \le i,i' \le I, \, 1 \le j \le N_i, \, 1 \le j' \le N_{i'}, 
\, 0 \le k,k' \le r$,
\begin{equation}
\label{2.36}
|\ln(F_{j,k}^i) - \ln(F_{j',k'}^{i'})| \le M|c_{j,k}^i- c_{j',k'}^{i'}|.
\end{equation}
and \eqref{2.36} implies that \eqref{2.35} holds.


One might hope to prove that the spectral radius $\rmf(\bL_s)$ of
$\bL_s$ closely approximates the spectral radius $\rmf(L_s)$, and we
shall see that this is true if the Lipschitz constant $C_0 \kappa$ in
(H2) (corresponding to the case $\nu=1$ and the operator $\rmf(L_s)$)
and the constant $h$ in \eqref{2.30} are sufficiently small.  However,
if $C_0 \kappa$ is not sufficiently small, we can instead work with
the operator $L_s^{\nu}$, where $\nu$ is a positive integer, and the
corresponding Lipschitz constant in (H2) is then $C_0 \kappa^{\nu}$. This
in turn means that we will have to replace $\bL_s$ by $\bL_{s,\nu}:
\R^{Q} \to \R^{Q}$, where, $\nu$ is a positive integer and in our
earlier notation,
\begin{equation}
\label{3.8}
(\bL_{s,\nu}(\bg))(c_{j,k}^i) = \sum_{\omega \in \Omega_{\nu}}
g_{\omega}(c_{j,k}^i)^s \cF(\theta_{\omega}(c_{j,k}^i)).
\end{equation}

\section{Cones, Positive eigenvectors, and Birkhoff's 
contraction constant}
\label{sec:theory-c}
As noted in Section~\ref{sec:intro}, we would like to have the
approximating matrices defined in the previous section mimic the
properties of the infinite dimensional, bounded linear operator $L_s$,
which means they should satisfy the conclusions of the Perron theorem
for positive matrices, i.e., they should have an eigenvalue of
multiplicity one equal to the spectral radius of the matrix with
corresponding positive eigenvector and that all other eigenvalues of
the matrix should have modulus less than the spectral radius.
However, the matrix $\bL_s$ defined in the previous section is not
even a nonnegative matrix once the degree $r$ of the piecewise
polynomial is $>1$.  The reason for this, seen by constructing the
matrix, is that the Lagrange basis functions for a polynomial of
degree $>1$ are not always positive.

The remedy, also used in the case $r=1$, when the resulting matrix was
nonnegative, but not {\it primitive} or {\it irreducible}, is to base
the analysis on a cone different from the usual cone of nonnegative
functions. More precisely, by using the cone $K_M$ defined in the
previous section, we shall show that the conclusions of the classical
Perron theorem also hold for the matrices of this paper.

To outline our method of proof, it is convenient to describe, at least
in the finite dimensional case, some basic definitions and classical
theorems concerning linear maps $L: \R^N \to \R^N$ which leave a cone
$K \subset \R^N$ invariant.  In doing so, we shall closely follow the
analogous description in \cite{hdcomp1}. Recall that a closed subset $K$ of
$\R^N$ is called a closed cone if (i) $ax + by \in K$ whenever $a
\ge 0$, $b \ge 0$, $x \in K$ and $y \in K$ and (ii) if $x \in
K \setminus \{0\}$, then $-x \notin K$.  If $K$ is a
closed cone, $K$ induces a partial ordering on $\R^N$ denoted by
$\le_{K}$ (or simply $\le$, if $K$ is obvious) by $u
\le_{K} v$ if and only if $v-u \in K$.  If $u,v \in K$, we
shall say that $u$ and $v$ are {\it comparable} (with respect to
$K$) and we shall write $u \sim_{K} v$ if there exist positive
scalars $a$ and $b$ such that $v \le_{K} au$ and $u \le_{K} b
v$. {\it Comparable with respect to} $K$ partitions $K$ into
equivalence classes of comparable elements.  We shall henceforth
assume that $\interior(K)$, the interior of $K$, is nonempty.
Then an easy argument shows that all elements of $\interior(K)$
are comparable.  Generally, if $x_0 \in K$ and $K_{x_0}: = \{x
\in K : x \sim_{K} x_0\}$, all elements of $K_{x_0}$ are
comparable.

Following standard notation, if $u,v \in K$ are comparable elements, we
define \begin{align*}
M(u/v;K) &= \inf\{\beta >0 : u \le \beta v\},
\\
m(u/v;K) &= M(v/u;K)^{-1} = \sup\{\alpha >0 : \alpha v \le u\}.
\end{align*}
If $u$ and $v$ are comparable elements of $K \setminus \{0\}$, we define
Hilbert's projective metric $d(u,v;K)$ by
\begin{equation*}
d(u,v;K) = \ln(M(u/v;K)) + \ln M(v/u;K)).
\end{equation*}
We make the convention that $d(0,0;K) =0$. If $x_0 \in K \setminus
\{0\}$, then for all $u,v,w \in K_{x_0}$, one can prove that (i)
$d(u,v;K) \ge 0$, (ii) $d(u,v;K) = d(v,u;K)$, and (iii)
$d(u,v;K) + d(v,w;K) \ge d(u,w;K)$. Thus $d$ restricted to
$K_{x_0}$ is almost a metric, but $d(u,v;K) =0$ if and only if $v =
tu$ for some $t >0$ and generally, $d(su,tv;K) = d(u,v;K)$ for all
$u,v \in K_{x_0}$ and all $s >0$ and $t >0$.  If $\|\cdot\|$ is any norm
on $\R^N$ and $S:= \{ u \in \interior(K): \|u\|=1\}$ (or, more generally,
if $x_0 \in K \setminus \{0\}$ and $S = \{x \in K_{x_0} : \|x\| =1\}$,
then $d(\cdot, \cdot; K)$, restricted to $S \times S$, gives a metric on
$S$; and it is known that $S$ is a complete metric space with this metric.

With these preliminaries, we can describe a special case of the
Birkhoff-Hopf theorem. We refer to \cite{Birkhoff}, \cite{V}, and
\cite{Samelson} for the original papers and to \cite{E-N-2} and
\cite{E-N-1} for an exposition of a general version of this theorem
and further references to the literature.  We remark that
P.~P.~Zabreiko, M.~A.~Krasnosel$'$skij, Y.~V.~Pokornyi, and
A.~V.~Sobolev independently obtained closely related theorems; and we
refer to \cite{Y} for details.  If $K$ is a closed cone as above, $S =
\{x \in \interior(K): \|x\|=1 \}$, and $L: \R^N \to \R^N$ is a linear
map such that $L(\interior(K)) \subset \interior(K)$, we define
$\Delta(L;K)$, {\it the projective diameter} of $L$ by
\begin{multline*}
\Delta(L;K) = \sup\{d(Lx,Ly;K): x,y \in K \text{ and }
 Lx \sim_{K} Ly\}
\\
= \sup\{d(Lx,Ly;K): x,y \in \interior(K)\}.
\end{multline*}
The Birkhoff-Hopf theorem implies that if $\Delta:= \Delta(L;K) < \infty$,
then $L$ is a contraction mapping with respect to Hilbert's projective metric.
More precisely, if we define $\lambda = \tanh(\tfrac{1}{4} \Delta) <1$, then
for all $x,y \in K \setminus \{0\}$ such that $x \sim_{K} y$, we have
\begin{equation*}
d(Lx,Ly;K) \le \lambda d(x,y;K),
\end{equation*}
and the constant $\lambda$ is optimal.

If we define $\Phi:S \to S$ by $\Phi(x) = L(x)/\|L(x)\|$, it follows that
$\Phi$ is a contraction mapping with a unique fixed point $v \in S$, and
$v$ is necessarily an eigenvector of $L$ with eigenvalue $r(L) := r=$
the spectral radius of $L$.  Furthermore, given any $x \in \interior(K)$,
there are explicitly computable constants $M$ and $c <1$ (see Theorem 2.1
in \cite{E-N-2}) such that for all $k \ge 1$,
\begin{equation*}
\|L^k(x)/\|L^k(x)\| -v\| \le Mc^k;
\end{equation*}
and the latter inequality is exactly the sort of result we need.  Furthermore,
it is proved in Theorem 2.3 of \cite{E-N-2} that $r=r(L)$ is an
algebraically simple eigenvalue of $L$ and that if $\sigma(L)$ denotes
the spectrum of $L$ and $q(L)$ denotes the {\it spectral clearance of } $L$,
\begin{equation*}
q(L):= \sup\{|z|/r(L): z \in \sigma(L), z \neq r(L)\},
\end{equation*}
then $q(L) <1$ and $q(L)$ can be explicitly estimated.

The main issue, then, is to find a suitable cone satisfying the hypotheses
outlined above.  We shall show in the sections that follow that the
cone $K_M$, defined in the previous section, is such a cone.  To do
so, we shall show that there exists $M^{\prime}$, $0 < M^{\prime} <
M$, such that $L(K_M\setminus\{0\}) \subset
K_{M^{\prime}}\setminus\{0\}$.  After correcting the typo in the
formula for $d_2(f,g)$ on page 286 of \cite{F}, it follows from Lemma
2.12 on page 284 of \cite{F} that
\begin{equation*}
\sup\{d(f,g;K_M) : f,g \in K_{M^{\prime}}\setminus\{0\}\} 
\le 2 \ln\Big(\frac{M + M^{\prime}}{M- M^{\prime}}\Big) 
+ 2 \exp(M^{\prime}(b-a)) < \infty,
\end{equation*}
where now $S:=[a, b]$ in (H1) and (H2) (c.f. Section~\ref{sec:approx}).
This implies that $\Delta(L;K_M) < \infty$, which in turn implies that
$L$ has a normalized eigenvector $v \in K_{M^{\prime}}$ with positive
eigenvalue $r = r(L) =$ the spectral radius of $L$.  
Furthermore, $r$ has algebraic multiplicity 1, $q(L) <1$, and
$\underset{k \to \infty}{\lim} \|L^k(x)/\|L^k(x)\| -v\| =0$ for all $x \in K_M
\setminus\{0\}$.  Thus it suffices to prove, for an appropriate map $L$,
that for some $M^{\prime} < M$,
\begin{equation}
\label{key-map-prop}
L(K_M\setminus\{0\}) \subset K_{M^{\prime}}\setminus\{0\}.
\end{equation}
We note that if the map $L$ satisfies \eqref{key-map-prop}, then it is
not difficult to show that $L(K_M\setminus\{0\})\subset \interior(K_M)$.
An alternative approach is then to apply Theorem 4.4 of
\cite{Vandergraft}, which concludes that $\rmf(L)$ is a simple
eigenvalue, greater than the magnitude of any other eigenvalue, and
that an eigenvector corresponding to $\rmf(L)$ lies in $\interior(K)$.  In any
case, the key step is to show for an appropriate matrix $L$ and cone
$K_M$, that \eqref{key-map-prop} is satisfied.

A key part of the paper is to obtain upper and lower bounds on
$R(L_s)$ using the approximations developed in this paper.  To do
so, we will use an extension to cones of a well known result for positive
matrices.
\begin{lem}
\label{lem:cone-compare}
Suppose $L(K_M\setminus\{0\}) \subset  K_{M^{\prime}}$ and 
$\cV_s \in K_M\setminus\{0\}$. Then if there exists positive constants
$\alpha$ and $\beta$ such that
\begin{equation*}
\alpha \cV_s \le_{K_M} L \cV_s \le_{K_M} \beta \cV_s,
\end{equation*}
then $\alpha \le R(L) \le \beta$.
\end{lem}

\section{Theory for the discrete approximation}
\label{sec:theory-d}
The main result of this section, Theorem~\ref{thm:5.6n}, is to show that
under appropriate hypotheses, the matrix operator $\bL_{s,\nu}$ defined
in Section~\ref{sec:approx}, satisfies
\begin{equation*}
\bL_{s,\nu}(K_M(T) \setminus \{0\}) \subset K_{M'}(T)\setminus \{0\},
\end{equation*}
where $T$ will be as in \eqref{T-def} below.

Throughout this section, we shall assume that (H1) and (H2) are satisfied
and that $S = [a,b]$ with $a <b$, where $S$ is as in (H1) and (H2). We shall
also assume that (H3), given below, is satisfied, and we shall use the 
notation of (H1), (H2), and (H3).

(H3): For a given positive integer $\nu$, there exist pairwise disjoint,
nonempty compact intervals $[a_i,b_i] \subset S$, $1 \le i \le I$,
(where $S:= [a,b]$ is as in (H0) - (H2)) with the following property:
For every $\omega \in \Omega_{\nu}$, there exists $i =i(\omega)$, $1 \le i \le
I$, such that  $\theta_{\omega}(S) \subset [a_{i},b_{i}]$.

\begin{remark}
\label{rem:disjoint} 
Assume that (H0)-(H2) are satisfied and that for some positive integer $\nu'$,
$\theta_{\omega_1}(S)$ and $\theta_{\omega_2}(S)$ are disjoint whenever
$\omega_1$ and $\omega_2$ are unequal elements of $\Omega_{\nu'}$.  Label the
elements of $\Omega_{\nu'}$ as $\omega_i$, $1 \le i \le I$, and define
$[a_i,b_i] = \theta_{\omega_i}(S)$.  Then for all positive integers $\nu \ge
\nu'$, (H3) is satisfied.  More generally, one could, for $1 \le i \le I$,
take $[a_i,b_i]$ to be any interval contained in $[a,b]$ such that
$\theta_{\omega_i}(S) \subset [a_i,b_i]$ as long as $[a_i,b_i] \cap [a_j,b_j]
= \emptyset$ for $1 \le i,j \le I$.  Note that (H3) is also trivially
satisfied if we take $I =1$ and $[a_1,b_1] = [a,b]$.
\end{remark}

\begin{remark}
\label{rem:ordered}
In the context of (H3), it is possible by relabeling to assume that
$b_i < a_{i+1}$ for $1 \le i < I$, so the intervals are linearly ordered
as subsets of $\R$. Thus we shall make this assumption if convenient.
\end{remark}

We now follow the notation of Section~\ref{sec:approx}.
If we define
\begin{equation}
\label{T-def}
T:=\{c_{j,k}^i : 1 \le i \le I, 1 \le j \le N_i, 0 \le k \le r\},
\end{equation}
then $T$ is a finite subset of $\R$ and a compact metric space.  Recall that
we consider the finite dimensional vector space $V= V(T)$ of dimension $Q = Nr
+I$ of all maps $F:T \to \R$ and $K_M(T)$ is then defined as in
Section~\ref{sec:theory-c} or \eqref{2.35}, i.e.,
$F \in K_M(T) \setminus \{0\}$ if and only if
\begin{equation*}
|\ln(F(\xi)) - \ln(F(\eta))| \le M|\xi - \eta|
\end{equation*}
for all $\xi, \eta \in T$.  Note that $V$ is linearly isomorphic to $\R^Q$.

{\it A central question for our approach
is to determine under what conditions on $\bL_{s,\nu}$ (see \eqref{3.8})
$\bL_{s,\nu}(K_M(T)) \subset K_{M'}(T)$ for
some $M$, $M'$ with $0 < M' < M$.} To do so, we begin by recalling some
classical results.

\begin{lem} (See \cite{Markov}, \cite{Rivlin}, pages 121-123, or
  \cite{Shadrin}).
\label{lem:3.1}
Let $p(t)$ be a real-valued polynomial of degree $\le r$. Then
\begin{equation*}
\max_{-1 \le t \le 1} |p^{\prime}(t)| \le r^2 \max_{-1 \le t \le 1} |p(t)|.
\end{equation*}
\end{lem}

A proof of the following estimate is given in \cite{Brutman} and refinements
for $r \ge 5$ can be found in \cite{Gunttner82}.
\begin{lem}
\label{lem:3.2}
If $\hat l_{k}(\hat x)$, $0 \le k \le r$, is defined by \eqref{2.32}, then
\begin{equation*}
 \max_{-1 \le \hat x \le 1} \sum_{k=0}^r|\hat l_{k}(\hat x)| 
\le \frac{2}{\pi} \ln(r +1) + 3/4 := \psi(r),
\end{equation*}
\end{lem}
where $\ln$ denotes natural logarithm.

It will also be convenient to have some elementary estimates concerning the
numbers $c_{j,k}^i$, $1 \le j \le N_i$, $0 \le k \le r$, in \eqref{2.33}.
If $x$ is a real number, $[x]$ denotes the greatest integer $m \le x$.
If $r$ is an integer, it follows that $[r/2] = r/2$ if $r$ is even and
$[r/2] = (r-1)/2$ is $r$ is odd.  The next lemma is a straightforward
exercise and is left to the reader.
\begin{lem}
\label{lem:3.3}
Let the numbers $c_{j,k}^i$ be defined by \eqref{2.33}. Then for $0 \le k \le
r$ and $1 \le i \le I$,
\begin{align*}
|c_{j,k}^i - c_{j,[r/2]}^i| &\le h_i/2, \quad \text{if } r \text{ is even},
\\
|c_{j,k}^i - c_{j,[r/2]}^i| &\le h_i/2[1 + \tan(\pi/[2r+2]), 
\quad \text{if } r \text{ is odd}.
\end{align*}
For $1 \le j \le N_i$ and $0 \le k < r$,
\begin{align*}
\min_{1 \le j \le N_i, 0 \le k < r} |c_{j,k}^i -c_{j,k+1}^i| &= 2h_i
[\sin(\pi/[2r+2])]^2,
\\
\max_{1 \le j \le N_i, 0 \le k < r} |c_{j,k}^i -c_{j,k+1}^i| &= h_i \tan(\pi/[2r+2])
\quad \text{if } r \text{ is odd},
\\
\max_{1 \le j \le N_i, 0 \le k < r} |c_{j,k}^i -c_{j,k+1}^i| &= h_i \sin(\pi/[2r+2])
\quad \text{if } r \text{ is even}.
\end{align*}
\end{lem}

Since the first result in Lemma~\ref{lem:3.3} will be used later, we
define the constant $\eta(r)$ for $r$ a
positive integer by:
\begin{equation}
\label{3.1}
\eta(r) = \begin{cases} \tfrac{1}{2}, & \text{if } r \text{ even}
\\
\tfrac{1}{2}[1 + \tan(\pi/[2r+2])], & \text{if } r \text{ odd}
\end{cases}.
\end{equation}
In addition, for $\nu$ a positive integer and
$\omega \in \Omega_{\nu}$, we define constants $M_0(\nu)$ and $c(\nu)$
such that for all $\omega \in \Omega_{\nu}$,
\begin{equation}
\label{3.4}
g_{\omega} \in K_{M_0(\nu)}(S) \qquad \text{and} \qquad \lip(\theta_{\omega}|_S)
\le c(\nu).
\end{equation}
We already know (see Section~\ref{sec:notation}) that $M_0(\nu) \le M_0 C_0 (1-
  \kappa^{\nu})/(1 - \kappa)$, where $M_0$ is as in (H1) and $C_0$ and
  $\kappa$ are as in (H2); and (H2) implies that $c(\nu) \le C_0
  \kappa^{\nu}$. However, in specific examples which we shall study later,
  these estimates can be significantly improved.

\begin{lem}
\label{lem:3.4}
Suppose that $\tau \in \R$, $\epsilon >0$, and $\hat c_k$, $0 \le k \le r$ is
a normalized extended Chebyshev point as in \eqref{hatck}, and $c_k = \tau + 
\tfrac{\epsilon}{2}(1 + \hat c_k)$. If $x \in [\tau, \tau + \epsilon]$,
let $\hat x \in [-1,1]$ denote the unique point such that 
$x = \tau + \tfrac{\epsilon}{2}[1 + \hat x]$. Assume that $M >0$,
$\Gamma = \{c_k: 0 \le k \le r\}$ and $F \in K_M(\Gamma) \setminus \{0\}$,
so $|\ln(F(\xi)) - \ln(F(\eta))| \le M|\xi - \eta|$ for all $\xi, \eta
\in \Gamma$. Let $\cF: [\tau, \tau + \epsilon] \to \R$ denote the unique
polynomial of degree $\le r$ such that $\cF(c_k) = F(c_k)$ for $0 \le k \le
r$. Let $\eta(r)$ be as in \eqref{3.1} and $\psi(r)$ as in Lemma~\ref{lem:3.2}
and define $u = M \epsilon \eta(r)$. If
\begin{equation*}
\psi(r) u \exp(u) <1,
\end{equation*}
then $\cF(x) >0$ for all $x \in [\tau, \tau + \epsilon]$. If 
$\psi(r) u \exp(u) <1$ and
\begin{equation*}
C:= \frac{[2 \eta(r) r^2 \psi(r)]\exp(u) M}{1 - \psi(r) u \exp(u)},
\end{equation*}
then for all $x,y \in [\tau, \tau + \epsilon]$,
\begin{equation*}
|\ln(\cF(x)) - \ln(\cF(y))| \le C|x-y|.
\end{equation*}
\end{lem}
\begin{proof}
Recall that for $0 \le k \le r$,
\begin{equation*}
l_k(x) = \frac{\prod_{\substack{l=0 \\ l \neq k}}^r (x-c_l)}
{\prod_{\substack{l=0 \\ l \neq k}}^r(c_k -c_l)}, \qquad
\hat l_k(\hat x) = \frac{\prod_{\substack{l=0 \\ l \neq k}}^r (\hat x- \hat c_l)}
{\prod_{\substack{l=0 \\ l \neq k}}^r(\hat c_k -\hat c_l)},
\end{equation*}
and $l_k(x) = \hat l_k(\hat x)$ for $x = \tau + \tfrac{\epsilon}{2}[1 + \hat
x]$ and, writing $F_k = F(c_k)$,
\begin{equation*}
\cF(x) = \sum_{k=0}^r l_k(x) F(c_k) = \sum_{k=0}^r l_k(x) F_k.
\end{equation*}

Recalling that $\sum_{k=0}^r l_k(x) =1$ for all
$x \in [\tau, \tau + \epsilon]$, we obtain
\begin{multline*}
\cF(x) = \sum_{k=0}^r l_k(x) F_k
= F_{[r/2]}\Big( 1  + \sum_{\substack{k=0 \\ k \neq [r/2]}}^r 
l_k(x) \Big[\frac{F_k}{F_{[r/2]}} - 1\Big]\Big)
= F_{[r/2]}[1 + \phi(x)]
\\
= F_{[r/2]}\Big( 1  + \sum_{\substack{k=0 \\ k \neq [r/2]}}^r 
\hat l_k(\hat x) \Big[\frac{F_k}{F_{[r/2]}} - 1\Big]\Big)
:= F_{[r/2]}[1 + \hat\phi(\hat x)],
\end{multline*}
where as usual, $x = \tau + \tfrac{\epsilon}{2}[1 + \hat x]$, $\hat x \in
[-1,1]$. 

Since $F \in K_M(\Gamma) \setminus \{0\}$, 
we have for $0 \le k \le r$, $k \neq [r/2]$,
\begin{equation*}
\exp(- M|c_k- c_{[r/2]}|) \le \frac{F_k}{F_{[r/2]}}
\le \exp(M|c_k- c_{[r/2]}|).
\end{equation*}
Because Lemma~\ref{lem:3.3} (with $h_i = \epsilon$) implies that $|c_k
- c_{[r/2]}| \le \eta(r) \epsilon$,
\begin{equation*}
\exp(-M \eta(r) \epsilon) -1 \le \frac{F_k}{F_{[r/2]}} -1 \le \exp(M \eta(r)
\epsilon) -1.
\end{equation*}
Using the mean value theorem and writing $u = M \eta(r) \epsilon$,
it follows that
\begin{equation*}
-u \le \frac{F_k}{F_{[r/2]}} -1
\le u \exp(u),
\end{equation*}
so
\begin{equation*}
\Big|\frac{F_k}{F_{[r/2]}} -1\Big|
\le u \exp(u).
\end{equation*}
Using Lemma~\ref{lem:3.2}, it follows that
\begin{equation}
\label{3.10}
|\hat \phi(\hat x)| \le
\sum_{\substack{k=0 \\ k \neq [r/2]}}^r |\hat l_k(\hat \xi)| 
\Big|\frac{F_k}{F_{[r/2]}}  -1\Big|
\le \psi(r) u \exp(u),
\end{equation}
so if $\psi(r) u \exp(u) <1$, $1 + \hat \phi(\hat x) > 0$, and $\cF(x) >0$
for all $x \in [\tau, \tau + \epsilon]$.
For the remainder of the proof, we assume that $\psi(r) u \exp(u) <1$.

If $x,y \in [\tau, \tau + \epsilon]$, our previous calculations show that
\begin{multline*}
|\ln \cF(x) - \ln \cF(y)|
= \Big|\ln [1  + \phi(x)]
- \ln [1  + \phi(y)] \Big|
\\
= \Big|\int_{1 + \phi(y)}^{1 + \phi(x)} \frac{1}{s} \, ds \Big|
\le 
\Big|\frac{\phi(x) - \phi(y)}{2}\Big| \Big[
\frac{1}{1 + \phi(x)} + \frac{1}{1 + \phi(y)}\Big]\Big|,
\end{multline*}
where we have used the fact that $(1/s)$ is a convex function and
hence the integral is bounded by the trapezoidal rule approximation.

Now, by the mean value theorem, for some $\hat \xi$ lying between $\hat x$ and
$\hat y$ and hence $\in [-1,1]$,
\begin{multline*}
|\phi(x) - \phi(y)| = |\hat \phi(\hat x) - \hat \phi(\hat y)|
= |\hat \phi^{\prime}(\hat \xi)| |\hat x- \hat y|
\\
\le \frac{2}{\epsilon} |x-y| \max_{-1 \le \hat \xi \le 1} 
|\hat \phi^{\prime}(\hat \xi)|
\le \frac{2}{\epsilon} |x-y| r^2  \max_{-1 \le \hat \xi \le 1} 
|\hat \phi(\hat \xi)|,
\end{multline*}
where in the last step we have used 
Markov's polynomial inequality (Lemma~\ref{3.1}).

Recalling our earlier estimate for $|\hat \phi(\hat \xi)|$ in \eqref{3.10},
we obtain
\begin{equation*}
|\phi(x) - \phi(y)| \le 2 r^2 \psi(r) \eta(r)
M \exp(u) |x-y|
\end{equation*}
and
\begin{equation*}
\frac{1}{2} \Big[\frac{1}{1+ \phi(x)} + \frac{1}{1+ \phi(y)} \Big]
\le \frac{1}{1 -  \psi(r) u \exp(u)},
\end{equation*}
which implies that
\begin{equation*}
|\ln( \cF(x)) - \ln(\cF(y))| \le C|x-y|,
\end{equation*}
with the constant $C$ defined in the statement of the lemma.
\end{proof}

\begin{lem}
\label{lem:5.5n}
Let notation be as in Section~\ref{sec:approx} and $T$ be as defined
by \eqref{T-def}. Suppose that $F:T \to \R$ is an element of $K_M(T)
\setminus \{0\}$ and let $\cF: \cup_{i=1}^I [a_i,b_i] \to \R$ be
defined by \eqref{2.31}. For $1 \le i \le I$, define $u_i = M h_i
\eta(r)$ and assume that
\begin{equation*}
\psi(r) u_i \exp(u_i) < 1.
\end{equation*}
Then $\cF(x) >0$ for all $x \in [a_i,b_i]$. If we define $C_i$ by
\begin{equation}
\label{5.13n}
C_i:= \frac{[2 \eta(r) r^2 \psi(r)]\exp(u_i) M}{1 - \psi(r) u_i \exp(u_i)},
\end{equation}
then for all $x,y \in [a_i,b_i]$,
\begin{equation}
\label{5.14n}
|\ln(\cF(x)) - \ln(\cF(y))| \le C_i|x-y|.
\end{equation}
\end{lem}
\begin{proof}
Recall that $[a_i,b_i] = \cup_{j=1}^{N_i} [t_{j-1}^i, t_j^i]$,
where $t_j^i - t_{j-1}^i = h_i = (b_i-a_i)/N_i$.  If we write $h_i =
\epsilon$, note that whenever $x,y \in  [t_{j-1}^i, t_j^i]$ for some $j$,
Lemma~\ref{lem:5.5n} implies that \eqref{5.14n} is satisfied. Thus we can
assume that $x,y \in [a_i,b_i]$ and that there does not exist $j$, $1 \le j
\le N_i$, such that $x$ and $y$ are both elements of $[t_{j-1}^i, t_j^i]$. We can
also assume that $x < y$ and select $j_1$, $1 \le j_1 \le N_i$, such that
$x \in [t_{j_1-1}^i, t_{j_1}^i]$ and  $j_2$, $1 \le j_2 \le N_i$, such that
$y \in [t_{j_2-1}^i, t_{j_2}^i]$. By our assumptions, it must be true that
$j_1 < j_2$,  If we apply Lemma~\ref{lem:5.5n} to $\cF(x)$ and
$\cF(t_{j_1}^i)$, we obtain
\begin{equation*}
|\ln(\cF(t_{j_1}^i)) - \ln(\cF(x))| \le C_i(t_{j_1}^i -x).
\end{equation*}
Similarly, if we apply Lemma~\ref{lem:5.5n} to $\cF(t_{j_2 -1}^i)$ and
$\cF(y)$, we obtain
\begin{equation*}
|\ln(\cF(y)) - \ln(\cF(t_{j_2 -1}^i))| \le C_i(y - t_{j_2 -1}^i).
\end{equation*}

Since $\cF(t_{j_1}^i) = F(t_{j_1}^i)$ and $\cF(t_{j_2 -1}^i) = F(t_{j_2
  -1}^i)$ and $F \in K_M(T)$, we obtain
\begin{equation*}
|\ln(\cF(t_{j_2 -1}^i)) - \ln(\cF(t_{j_1}^i))| \le M(t_{j_2 -1}^i - t_{j_1}^i)
\le C_i (t_{j_2 -1}^i - t_{j_1}^i),
\end{equation*}
where we have used the fact that $C_i >M$. Combining these inequalities, we
find that
\begin{multline*}
|\ln(\cF(y)) - \ln(\cF(x))| \le |\ln(\cF(y)) - \ln(\cF(t_{j_2 -1}^i))|
\\
+ |\ln(\cF(t_{j_2 -1}^i)) - \ln(\cF(t_{j_1}^i))|
+ |\ln(\cF(t_{j_1}^i)) - \ln(\cF(x))| \le C_i (y-x),
\end{multline*}
which proves Lemma~\ref{lem:5.5n}.
\end{proof}

Up to this point, we have only used the fact that $F \in K_M(T)$,
where $T$ is defined in \eqref{T-def} and notation is as in
Section~\ref{sec:approx}.  We now exploit the fact that $\lip(\theta_{\omega}|_S)
\le c(\nu)$ for all $\omega \in \Omega_{\nu}$.


\begin{thm}
\label{thm:5.6n}
Let notation be as in Section~\ref{sec:approx} and for positive reals $M' <
M$, let $K_M(T)$ and $K_{M'}(T)$ be as defined earlier, Recall that $h_i =
(b_i-a_i)/N_i$, $1 \le i \le I$, and $h = \max \{h_i: 1 \le i \le I\}$.
Assume that hypotheses (H1), (H2), and (H3) are satisfied, and that
\begin{equation}
\label{5.16n}
\psi(r) u \exp u <1,
\end{equation}
where we now set 
\begin{equation*}
u = M h \eta(r).
\end{equation*}
If $F \in K_M(T) \setminus \{0\}$ and $\cF$ is the piecewise polynomial
approximation of $F$ of degree $\le r$ on $\hat S = \cup_{i=1}^I [a_i,b_i]$,
then $\cF(x) >0$ for all $x \in \hat S$.

Define $C:= \max \{C_i: 1 \le i \le I\}$, where $C_i$ is defined
by \eqref{5.13n} and let $\bL_{s, \nu}: V(T) := \R^Q \to V(T)$ be defined
by \eqref{3.8}. Assume the above hypotheses are satisfied and also assume
that
\begin{equation}
\label{5.17n}
c(\nu)C :=
\frac{c(\nu)[2 \eta(r) r^2 \psi(r)] \exp(u) M}
{1 - \psi(r) u \exp(u)}< M' - s M_0(\nu).
\end{equation}
Then it follows that
$\bL_{s,\nu}(K_M(T) \setminus \{0\}) \subset K_{M'}(T)\setminus \{0\}$.
\end{thm}
\begin{proof}
Suppose that $F \in K_M(T) \setminus \{0\}$, which implies that $F(\xi) >0$ for
all $\xi \in T$.  Since $u \ge u_i$ for $1 \le i \le I$, \eqref{5.16n}
implies that $\psi(r) u_i \exp u_i <1$ for $1 \le i \le I$. It follows
from Lemma~\ref{lem:5.5n} that $\cF(x) >0$ for all $x \in [a_i,b_i]$,
$1 \le i \le I$, so $\cF(x) >0$ for all $x \in \hat S$.  Since 
(H1) implies that $g_{\omega}(x)^s >0$ for all $x \in S = [a,b]$ and
for all $\omega \in \Omega_{\nu}$
\begin{equation*}
\bL_{s,\nu}(F)(\xi) = \sum_{\omega \in \Omega_{\nu}} [g_{\omega}(\xi)]^s
\cF(\theta_{\omega}(\xi)),
\end{equation*}
and $g_{\omega}(x)^s  \cF(\theta_{\omega}(x)) >0$ for all $x \in \hat S$ and 
certainly for all $\xi \in T$, it suffices to prove that the map $\xi \mapsto
g_{\omega}(\xi)^s \cF(\theta_{\omega}(\xi)) \in K_{M'}(T) \setminus \{0\}$ for
every $\omega \in \Omega_{\nu}$.  We know that for all $x,y \in S$,
\begin{equation*}
|\ln(g_{\omega}(x)^{s}) - \ln(g_{\omega}(y)^{s})| \le s M_0(\nu)|x-y|,
\end{equation*}
so it suffices to prove that for all $x, y \in \hat S$,
\begin{equation*}
|\ln(\cF(\theta_{\omega}(x)) - \ln(\cF(\theta_{\omega}(y))| \le [M' - s
M_0(\nu)] \, |x-y|.
\end{equation*}
For each fixed $\omega \in \Omega_{\nu}$, (H3) implies that there exists
$i = i(\omega)$ such that $\theta_{\omega}(\hat S) \subset [a_i,b_i]$.
Writing $x' = \theta_{\omega}(x) \in [a_i,b_i]$ and
 $y' = \theta_{\omega}(y) \in [a_i,b_i]$, Lemma~\ref{lem:5.5n} implies that
\begin{equation*}
|\ln(\cF(x')) - \ln(\cF(y'))| \le C|x'-y'| \le c(\nu)C|x-y|,
\end{equation*}
so \eqref{5.17n} completes the proof.
\end{proof}

\begin{remark}
\label{rem:3.2} 
Assume that $\psi(r) u \exp(u) <1$.
Notice that for a given positive integer $r$, a necessary
condition that \eqref{5.17n} be satisfied is that
\begin{equation}
\label{3.21}
c(\nu) 2 r^2 \psi(r) \eta(r) < \frac{M' -s M_0(\nu)}{M}.
\end{equation}
For a given $M' < M$, if \eqref{3.21} is satisfied, then \eqref{5.17n} will be
satisfied if $h$ is sufficiently small.
\end{remark}

\begin{remark}
\label{rem:5.3n}
The reader will note that in our definition of $\bL_{s,\nu}(F)$ for
$F \in V(T)$, we arrange that $\cF|_{[a_i,b_i]}$ is a piecewise polynomial
map of degree $\le r$.  In some applications, it is desirable to make
$\cF|_{[a_i,b_i]}$ a piecewise polynomial map of degree $\le r_i$, which
leads to a generalization of the definition of $\bL_{s,\nu}$. An analogue
of Theorem~\ref{thm:5.6n} which handles this more general case can be
proved by an argument similar to the proof of Theorem~\ref{thm:5.6n}.
Because of considerations of length, we omit the proof.
\end{remark}

\section{Estimating $\rmf(L_s)$ by the spectral radius of 
$\bL_{s,\nu}$}
\label{sec:estimating-sr}
In the previous section (c.f. Theorem~\ref{thm:5.6n}), 
we determined conditions under which
\begin{equation*}
\bL_{s,\nu}(K_M(T) \setminus \{0\}) \subset K_{M'}(T)\setminus \{0\}
\end{equation*}
for some $M' < M$. The main result of this section is to show that
under this condition, $\rmf(\bL_{s, \nu})$, the spectral radius of 
$\bL_{s, \nu}$, satisfies
\begin{equation*}
\lambda_s^{\nu}(1 - H h^r) \le \rmf(\bL_{s,\nu}) \le \lambda_s^{\nu}(1 + H h^r)
\end{equation*}
for some constant $H$ to be specified below. Using this result, we
obtain the following explicit bounds on the spectral radius
$\lambda_s$ of $L_s$.
\begin{equation*}
[\rmf([1 + H h^r]^{-1} \bL_{s,\nu})]^{1/\nu} \le \lambda_s \le
[\rmf([1 - H h^r]^{-1} \bL_{s,\nu})]^{1/\nu},
\end{equation*}
where the entries of the matrices $[1 + H h^r]^{-1} \bL_{s,\nu}$ and
$[1 - H h^r]^{-1} \bL_{s,\nu}$ differ by $O(h^r)$.

Throughout this section, we shall assume the hypotheses and notation of
(H1), (H2), and (H3); and $v_s(\cdot)$ will always denote the positive
eigenvector $v_s$ of $L_s$ guaranteed by Theorem~\ref{thm:2.4}. In particular,
$S$ and $[a_i,b_i]$, $1 \le i \le I$ will be as in (H3) and
(as can be guaranteed by relabeling), we shall assume that
$b_i < a_{i+1}$ for $1 \le i < I$.

We shall further denote by $E$ and $\chi$, constants for which the following
two inequalities are satisfied.
\begin{equation}
\label{6.1}
\sup_{a \le x \le b} \frac{|d^p v_s(x)/dx^p|}{v_s(x)} \le E(s,p):=E,
\end{equation}
where $p$ is a positive integer, and
\begin{equation}
\label{6.2}
v_s(x_1) \le v_s(x_2) \exp(2s|x_1-x_2|/\chi),
\end{equation}
and for all $x_1,x_2 \in S$, where $\chi:= \chi(s, \{\theta_i, g_i:1
\le i \le n\})$.
Using Theorem~\ref{thm:2.4} and Remark~\ref{rem:2.5}, we shall see, in the
next section, that for some interesting examples, it is possible to
obtain sharp estimates on the constants $E$ and $\chi$ such that \eqref{6.1}
and \eqref{6.2} below are satisfied. These estimates will refine
earlier results in \cite{hdcomp1}.

Using the notation of Section~\ref{sec:approx}, if $\cV_s(x)$ is the piecewise
polynomial interpolant of $v_s(x)$ of degree $\le r$ at the extended Chebyshev
points in $[a_i,b_i]$, $1 \le i \le I$, then on each subinterval
$[t_{j-1}^i,t_j^i]$, $j = 1, \ldots, N_i$, we have, using standard estimates for
polynomial interpolation,
\begin{equation*}
v_s(x) - \cV_s(x) =  \frac{v_s^{(r+1)}(\xi_x)}{(r+1)!}
\prod_{k=0}^r (x - c_{j,k}^i),
\end{equation*}
for some $\xi_x \in [t_{j-1}^i,t_j^i]$.
If we write, as done previously,
\begin{equation*}
c_{j,k}^i = t_{j-1}^i + h_i(1+ \hat c_k)/2, \qquad
x = t_{j-1}^i + h_i(1 + \hat x)/2, \quad \hat x \in [-1,1],
\end{equation*}
then for $x \in [a_i,b_i]$, 
\begin{equation*}
|v_s(x) - \cV_s(x)| \le  |v_s^{(r+1)}(\xi_x)| h_i^{r+1} m_{r+1},
\end{equation*}
where
\begin{equation}
\label{6.4}
m_{r+1} = \frac{1}{2^{r+1} (r+1)!} 
\max_{\hat x \in [-1,1]} \Big|\prod_{k=0}^r (\hat x - \hat c_k)\Big|.
\end{equation}
Using \eqref{6.1} and \eqref{6.2}, we see that
\begin{equation*}
|v_s^{(r+1)}(\xi_x)| \le E \exp(2sh_i/\chi) v_s(x),
\end{equation*}
so
\begin{equation*}
|v_s(x) - \cV_s(x)| \le  E h_i^{r+1} m_{r+1} v_s(x) \exp(2s h_i/\chi).
\end{equation*}
Defining, for $1 \le i \le I$,
\begin{equation}
\label{6.6}
G_{r,i} := E m_r \exp(2s h_i/\chi),
\end{equation}
\eqref{6.4} implies that for $1 \le i \le I$ and $x \in [a_i,b_i]$
\begin{equation}
\label{6.7}
(1 - G_{r+1,i} h_i^{r+1}) v_s(x) \le \cV_s(x) \le (1 + G_{r+1,i}  h_i^{r+1}) v_s(x).
\end{equation}
In order to make \eqref{6.4} explicit, we need a formula for
$\max_{\hat x \in [-1,1]} \Big|\prod_{k=0}^r (\hat x - \hat
c_k)\Big|$.  The result and proof, which we provide below, are slight
modifications of the well-known corresponding bound and proof when
$\hat c_k$ are taken to be the zeros of the standard Chebyshev
polynomial.
\begin{lem}
\label{lem:6.1}
If $r \ge 2$ is a positive integer and $\hat c_k$ is defined by \eqref{hatck},
then
\begin{equation*}
\max_{\hat x \in [-1,1]} \Big|\prod_{k=0}^r (\hat x - \hat c_k)\Big|
= \frac{1}{2^r} \Big[\frac{1}{\cos(\pi/[2r+2])}\Big]^{r+1}.
\end{equation*}
\end{lem}
\begin{proof}
If we define $w = \hat x \cos(\pi/[2r+2])$, where $|\hat x| \le 1$, we have
$|w| \le \cos(\pi/[2r+2])$. For notational convenience, we write
$\alpha = 1/ \cos(\pi/[2r+2])$, and we obtain
\begin{multline*}
\max_{\hat x \in [-1,1]} \Big|\prod_{k=0}^r (\hat x - \hat c_k)\Big|
\\
= \alpha^{r+1}  \max \Big\{ \Big|\prod_{k=0}^r(w + \cos([2k+1]\pi/[2r+2]))\Big|:
|w| \le \cos(\pi/[2r+2])\Big\}
\\
=  \alpha^{r+1}  \max \Big\{ \Big|\prod_{k=0}^r(w - \cos([2k+1]\pi/[2r+2]))\Big|:
|w| \le \cos(\pi/[2r+2])\Big\}.
\end{multline*}
If we define $q_{r+1}(w) = \prod_{k=0}^r(w -  \cos([2k+1]\pi/[2r+2]))$, 
 $q_{r+1}(w)$ is a polynomial of degree $r+1$ which vanishes at the points
$\cos([2k+1]\pi/[2r+2])$, $0 \le k \le r$, and has leading term $w^{r+1}$; and
these properties uniquely determine $q_{r+1}(w)$.

Recall that for integers $r \ge 0$, $\cos([r+1]\theta) = p_{r+1}(\cos
\theta)$ for $0 \le \theta \le \pi$, where $p_{r+1}(w)$ is the {\it
  Chebyshev} polynomial of degree $r+1$. These polynomials satisfy
$p_1(w) = w$, $p_2(w) = 2 w^2 -1$, and for $r \ge 2$, the recurrence
relation $p_{r+1}(w) = 2w p_r(w) - p_{r-1}(w)$.
Using the recursion relation for $p_{r+1}(w)$ and induction, it also
follows that the coefficient of $w^{r+1}$ in $p_{r+1}(w)$ is
$2^r$. Since $\cos([r+1]\theta) =0$ when $\theta =
[2k+1]\pi/[2r+2]$ for $0 \le k \le r$, we see that $p_{r+1}(w)
=0$ when $w = \cos([2k+1]\pi/[2r+2])$, for $0 \le k \le r$. It follows that
for $w = \cos(\theta)$ and $0 \le \theta \le \pi$,
\begin{equation*}
q_{r+1}(w)= \frac{1}{2^r} p_{r+1}(w) =  \frac{1}{2^r} \cos([r+1] \theta),
\end{equation*}
so
\begin{equation*}
\max_{|w| \le 1} |q_{r+1}(w)| = 1/2^r.
\end{equation*}
However,
\begin{multline*}
\max \{ |q_{r+1}(w)|: |w| \le \cos(\pi/[2r+2])\}
\\
 = \frac{1}{2^r} \max \{|\cos([r+1]\theta)|: \pi/[2(r+1)^2] \le \theta
\le (2r+1) \pi/(2r+2)\}.
\end{multline*}
Since $\pi/[2(r+1)^2] < \pi/(r+1) <  (2r+1) \pi/(2r+2)$ and
$|\cos([r+1] \pi/(r+1))| =1$,
\begin{equation*}
 \frac{1}{2^r}\max \{ |q_{r+1}(w)|: |w| \le \cos(\pi/[2r+2])\} = \frac{1}{2^r},
\end{equation*}
which completes the proof.
\end{proof}


If $E$ and $\chi$ are defined by \eqref{6.1} and \eqref{6.2}, we can use
Lemma~\ref{lem:6.1} to estimate the constant $G_{r+1,i}$ in \eqref{6.6} more
precisely:
\begin{equation}
\label{6.9}
G_{r+1,i} = E \exp(2sh_i/\chi) \Big[\frac{1}{(r+1)!}\Big] 
\Big[\frac{1}{2 \cos(\pi/[2r+2])}\Big]^{r+1} \frac{1}{2^r},
\end{equation}
and with this estimate of $G_{r+1,i}$, \eqref{6.7} is satisfied
for all $x \in [a_i,b_i]$, $1 \le i \le I$.  For notational convenience,
we define $h$ and $G_{r+1}$ by $h = \max_{1 \le i \le I} h_i$ and
\begin{equation}
\label{rn:6.8}
G_{r+1} = \max_{1 \le i \le I} G_{r+1,i}
= E \exp(2sh/\chi) \Big[\frac{1}{(r+1)!}\Big] 
\Big[\frac{1}{2 \cos(\pi/[2r+2])}\Big]^{r+1} \frac{1}{2^r}.
\end{equation}


\begin{lem}
\label{lem:6.3}
Define $\lambda_s = \rmf(\Lambda_s) = \rmf(L_s)$, where $\Lambda_s$
and $L_s$ are as in Theorem~\ref{thm:2.4}. Let $[a_i,b_i]$, $1 \le i
\le I$, be as in (H3). For $1 \le i \le I$, let $N_i$, $h_i$ and
$\cV_s$ be as defined in the fourth paragraph of this section.  Assume
that, for $1 \le i \le I$,
\begin{equation*}
[\sin(\pi/(2r+2)]^2 h_i \le a_{i+1} - b_i.
\end{equation*}
Define $h_{min} = \min_{1 \le i \le I} h_i$ and $\mu = h/h_{min}$.
Then we have for all $x \in [a_i,b_i]$, $1 \le i \le I$,
\begin{equation}
\label{6.10}
(1 - G_{r+1,i} h_i^{r+1}) v_s(x) \le \cV_s(x) \le (1 + G_{r+1,i} h_i^{r+1}) v_s(x),
\end{equation}
and
\begin{equation}
\label{6.11}
(1 - G_{r+1} h^{r+1}) \lambda_s^{\nu} v_s(x) \le (L_s^{\nu} \cV_s)(x)
\le (1 + G_{r+1}  h^{r+1}) \lambda_s^{\nu} v_s(x).
\end{equation}
If we define $M_1$ by
\begin{equation}
\label{6.12}
M_1 = \Big[\mu \frac{G_{r+1} h^r}{1- G_{r+1}^2 h^{2r+2}}\Big] 
\Big[\frac{1}{\sin(\pi/[2r+2])}\Big]^2 + \frac{2s}{\chi}
\end{equation}
and 
\begin{equation*}
T = \{c_{j,k}^i: 1 \le i \le I, 1 \le j \le N_i, 0 \le k \le r\} \subset S,
\end{equation*}
then $\cV_s|_T \in K(2s/\chi;T)$ and $L_s^{\nu} \cV_s|_T
= \bL_{s,\nu}(\cV_s|_T) \in K(M_1;T)$.
\end{lem}
\begin{proof}
  To simplify the exposition, we shall denote $G_{r+1}$ as $G$.  Equation
  \eqref{6.7} gives \eqref{6.10} and \eqref{6.2} implies that $v_s \in
  K(2s/\chi;S)$.  Since $ \cV_s|_T = v_s|_T$, it follows that $\cV_s|_T \in
  K(2s/\chi;S)$. If we observe that $1 - Gh^{r+1} \le 1 - G_{r+1,i}h_i^{r+1}$
  and $1 + G_{r+1,i}h_i^{r+1} \le 1 + Gh^{r+1}$ for $1 \le i \le I$, we derive
  from \eqref{6.7} that for  $1 \le i \le I$ and $x \in [a_i,b_i]$,
\begin{equation*}
(1 - Gh^{r+1}) v_s(x) \le \cV_s(x) \le (1 + Gh^{r+1}) v_s(x).
\end{equation*}
Applying $L_s^{\nu}$ to this inequality, we obtain \eqref{6.11} and in
particular, \eqref{6.11} holds for all $x \in T$.  A little thought
shows that for all $x \in T$,
\begin{equation*}
[L_s^{\nu} \cV_s](x) = (\bL_{s,\nu}(\cV_s|_T)(x).
\end{equation*} 

If $x,y \in T \cap [a_i,a_{i+1}]$, $1 \le i \le I$, and $x \neq y$, we obtain
from \eqref{6.11} that
\begin{multline*}
(L_s^{\nu}\cV_s)(x) \le (1 + G h^{r+1}) \lambda_s^{\nu} v_s(x)
\le (1 + G h^{r+1}) \exp(2s|x-y|/\chi) v_s(y)
\\
\le \frac{1 + G h^{r+1}}{1 - G h^{r+1}}\exp(2s|x-y|/\chi) (L_s^{\nu}\cV_s)(y) .
\end{multline*}
Taking logarithms on both sides of the above inequality, and noting that
$x$ and $y$ are interchangeable in the inequality, we find that
\begin{equation*}
|\ln(L_s^{\nu}\cV_s)(x) - \ln(L_s^{\nu}\cV_s)(y)|
\le \frac{2s}{\chi}|x-y| + [\ln(1 + G h^{r+1}) - \ln(1 - G h^{r+1})].
\end{equation*}

Using the trapezoidal rule and the convexity of $u \mapsto 1/u$,
\begin{multline*}
\ln (1 + G h^{r+1}) - \ln(1 - G h^{r+1})  = \int_{1 - G h^{r+1}}^{1 + G h^{r+1}}
\frac{1}{u} \, du
\\
\le \frac{1}{2}\Big[\frac{1}{1 - G h^{r+1}} + \frac{1}{1 + G h^{r+1}}\Big]
\Big[2G h^{r+1}\Big] = \frac{2G h^{r+1}}{1 - G^2 h^{2r+2}}.
\end{multline*}
To prove that $L_s^{\nu}\cV_s|_T \in K(M_1;T)$, it will suffice to prove that
\begin{multline}
\label{6.13n}
\frac{2s}{\chi}|x-y| + \frac{2G h^{r+1}}{1 - G^2 h^{2r+2}} 
\\
\le \frac{2s}{\chi}|x-y| + \Big[\mu \frac{Gh^r}{1- G^2 h^{2r+2}}\Big] 
\Big[\frac{1}{\sin(\pi/[2r+2])}\Big]^2 |x-y|,
\end{multline}
whenever $x,y  \in ([a_i,b_i] \cap T) \cup \{a_{i+1}\}$ for $1 \le i \le I$
and whenever $x,y  \in ([a_I,b_I] \cap T)$. (Of course, we assume, as we can,
that $x \neq y$). A calculation shows that this will be true if
\begin{equation*}
2h \le \mu \Big[\frac{1}{\sin(\pi/[2r+2])}\Big]^2 |x-y|.
\end{equation*}
If $x,y \in [a_i,b_i]$, we know that $|x-y| \ge 2 h_i [\sin(\pi/(2r+2)]^2$,
so it suffices to prove that $h \le \mu h_i$, which follows from the
definition of $\mu$.  We can assume that $x <y$, so if $x,y  \in [a_i,b_i]$,
the same argument applies.  If $y = a_{i+1}$, $|x-y| \ge |a_{i+1} - b_i|$,
and we assume that $|a_{i+1} - b_i| \ge 2 h_i  [\sin(\pi/(2r+2)]^2$, so
again the same argument applies and gives \eqref{6.13n}.
\end{proof}

\begin{remark}
\label{rem:6.1n}
If $I=1$, the condition on $a_{i+1} - b_i$ is vacuous and $\mu =1$.
\end{remark}

Our next lemma will play a crucial role in relating $\rmf(\bL_{s,\nu})$ to
$\rmf(L_s^{\nu})$.
\begin{lem}
\label{lem:6.4}
Let notation and assumptions be as in Lemma~\ref{lem:6.3}.  Let $G:=G_{r+1}$
be as in \eqref{rn:6.8} and $M_1$ as in \eqref{6.12}.  Assume
that $H:= H_{r+1}$ is a constant with $H >G$ and assume that $h <1$.  Define
$M_2$ by
\begin{equation*}
M_2 = M_1 + \frac{G}{H} \Big[\frac{\mu}{1 - [(G/H)h]^2}\Big]
\frac{1}{[\sin(\pi/[2r+2])]^2}.
\end{equation*}
If $K = K(M_2;T)$, we have
\begin{equation}
\label{6.13}
\lambda_s^{\nu} \cV_s (1 - H h^r) \le_{K} \bL_{s, \nu} \cV_s|_T \le_{K}
\lambda_s^{\nu} \cV_s (1 + H h^r).
\end{equation}
\end{lem}
\begin{proof}
Our previous results show that $(L_s^{\nu} \cV_s)(x)
= (\bL_{s,\nu} \cV_s)(x)$ for all $x \in T$, and, for all $x \in S = [a,b]$,
\begin{equation*}
\lambda_s^{\nu} \cV_s(x) \frac{1 - G h^{r+1}}{ 1 + G h^{r+1}}
\le (L_s^{\nu} \cV_s)(x) \le
\lambda_s^{\nu} \cV_s(x) \frac{1 + G h^{r+1}}{ 1 - G h^{r+1}}.
\end{equation*}
Recalling that $\cV_s(x) = v_s(x)$ for $x \in T$, we have for $x \in T$,
\begin{multline*}
\lambda_s^{\nu} \cV_s(x)(1 + H h^r) - (L_s^{\nu} \cV_s)(x)
\le \lambda_s^{\nu} \cV_s(x)(1 + H h^r) -  \lambda_s^{\nu}(1- G h^{r+1}) 
\cV_s(x)
\\
= \lambda_s^{\nu}(H h^r + Gh^{r+1})\cV_s(x)
= \lambda_s^{\nu} h^r \cV_s(x)(1 + [G/H] h) H.
\end{multline*}
If $y \in T$, a similar argument
shows that
\begin{multline*}
\lambda_s^{\nu} \cV_s(y)(1 + H h^r) - (L_s^{\nu} \cV_s)(y)
\ge \lambda_s^{\nu} \cV_s(y)(1 + H h^r) -  \lambda_s^{\nu}(1+ G h^{r+1}) 
\cV_s(y)
\\
= \lambda_s^{\nu} h^r \cV_s(y)(1 - [G/H] h) H.
\end{multline*}
Using Lemma~\ref{lem:6.3} and the above estimates, we find that
\begin{multline}
\label{6.14}
\frac{\lambda_s^{\nu} \cV_s(x)(1 + H h^r) - (L_s^{\nu} \cV_s)(x)}
{\lambda_s^{\nu} \cV_s(y)(1 + H h^r) - (L_s^{\nu} \cV_s)(y)}
\le \frac{\cV_s(x)(1 + [G/H] h)}{\cV_s(y)(1 - [G/H] h)}
\\
\le \exp(M_1|x-y|) \frac{1 + [G/H] h}{1 - [G/H] h}.
\end{multline}
The right half of \eqref{6.13} will follow from \eqref{6.14} if we prove
that, for all $x,y \in T$ with $x \neq y$,
\begin{equation}
\label{6.16n}
 \exp(M_1|x-y|) \frac{1 + [G/H] h}{1 - [G/H] h}
\le \exp(M_2|x-y|).
\end{equation}
As in Lemma~\ref{lem:6.3}, it suffices to verify \eqref{6.16n} for all
points $x \neq y$, $x,y \in [a_i,a_{i+1}] \cap T$, $1 \le i \le I$, where
$a_{I+1} = b_I$.

Arguing as in Lemma~\ref{lem:6.3}, we see that
\begin{equation*}
\ln(1 + [G/H] h) - \ln(1 - [G/H] h) \le \frac{G}{H} 
\frac{2h}{1 - ([G/H] h)^2}.
\end{equation*}
If we take the log of both sides of \eqref{6.16n}, it suffices to prove
that
\begin{multline*}
M_1 |x-y| + \frac{G}{H} \frac{2h}{1 - ([G/H] h)^2}
\\
\le M_1 |x-y| + \frac{G}{H}  \Big[\frac{\mu}{1 - [(G/H)h]^2}\Big]
\frac{1}{[\sin(\pi/[2r+2])]^2} |x-y|.
\end{multline*}
As was proved in Lemma~\ref{lem:6.3}, all $x,y \in [a_i,a_{i+1}] \cap T$
with $x \neq y$ satisfy $|x-y| \ge 2 h_i  [\sin(\pi/(2r+2)]^2$.
Since $2|x-y| \ge 2h_i/[\sin(\pi/[2r+2])]^2$, we see
after simplification, the above inequality will be satisfied
if $h \le \mu h_i$, which holds by the definition of $\mu$.  This proves
the right hand side of \eqref{6.13}. The proof of the left hand side 
of inequality \eqref{6.13} follows by an exactly analogous argument and
is left to the reader.
\end{proof}

Our next theorem connects $\rmf(\bL_{s,\nu})$ and $\rmf(L_s^{\nu})$.  To use the
theorem, we shall need to estimate various constants, and we shall carry
this out in the next section for an important class of examples.

\begin{thm}
\label{thm:6.5}
Let notation and assumptions be as in Lemma~\ref{lem:6.3} and let
$H$ and $M_2$ be as in Lemma~\ref{lem:6.4}. Assume that $\nu, h, s$ and
$r$ have been selected so that $(\bL_{s,\nu}(K(M;T)) \subset K(M';T)$ where
$0 < M' < M$ and $M \ge M_2$ (see Theorem~\ref{thm:5.6n}). Then we have
that $\rmf(\bL_{s, \nu})$, the spectral radius of $\bL_{s, \nu}$, satisfies
\begin{equation*}
\lambda_s^{\nu}(1 - H h^r) \le \rmf(\bL_{s,\nu}) \le \lambda_s^{\nu}(1 + H h^r).
\end{equation*}
\end{thm}
\begin{proof}
  Our previous results show that $\bL_{s,\nu}$ has a unique, strictly
  positive eigenvector $w_{s,\nu} \in K(M;T)$ with $\|w_{s,\nu}\|=1$.  The
  eigenvalue corresponding to $w_{s,\nu}$ is $\rmf(\bL_{s,\nu})$.
  Furthermore, for every $u \in K(M;T)\setminus\{0\}$, $\lim_{m \rightarrow
    \infty} (\bL_{s,\nu}^m u/\|\bL_{s,\nu}^m u\|) = w_{s,\nu}$, with
  convergence in the $\sup$ norm topology on $\R^Q$.

If we use \eqref{6.13}, but define $K = K(M;T)$, then because $M \ge M_2$ and
$K(M;T) \supset K(M_2;T)$, we obtain
\begin{equation*}
\lambda_s^{\nu} \cV_s (1 - H h^r) \le _K  \bL_{s,\nu} \cV_s 
\le _K \lambda_s^{\nu} \cV_s (1 + H h^r).
\end{equation*}
The theorem now follows directly from Lemma~\ref{lem:cone-compare}.
\end{proof}

Our ultimate goal has been to provide rigorous upper and lower bounds on
$\lambda_s = \rmf(L_s)$ in terms of the eigenvalues of computable matrices, as
was done in \cite{hdcomp1}. This follows immediately from
Theorem~\ref{thm:6.5}.
\begin{thm}
\label{thm:6.6}
Under the hypotheses of Theorem~\ref{thm:6.5}, we have
\begin{equation*}
[(1 + H h^r)^{-1} \rmf(\bL_{s,\nu})]^{1/\nu} \le \lambda_s \le
[(1 - H h^r)^{-1} \rmf(\bL_{s,\nu})]^{1/\nu},
\end{equation*}
where the entries of the matrices $[1 + H h^r]^{-1} \bL_{s,\nu}$ and
$[1 - H h^r]^{-1} \bL_{s,\nu}$ differ by $O(h^r)$.
\end{thm}

\section{Calculating the optimal interval $[a,b]$ and 
estimating $E$ and $\chi$.}

\label{sec:est-constants}
Throughout this section, we shall assume at least the hypotheses of
Theorem~\ref{thm:2.4}, so $L_s$ has a strictly positive $C^m$ eigenfunction
$v_s$.  We shall take $S=[a,b], a<b$ in (H2). If $S_0$ is a closed, nonempty
subset of $S$ and $\theta_i(S_0) \subset S_0$ for $1 \le i \le n$, then
$L_s:C(S) \to C(S)$ induces a bounded linear operator $L_{s,S_0}: C(S_0) \to
C(S_0)$.  It is often desirable to replace the original interval $[a,b]$ by a
smaller interval (or union of intervals) $S_0 \subset [a,b]$ such that
$\theta_i(S_0) \subset S_0$ for $1 \le i \le n$, and we first describe a class
of examples for which this can be easily done. Note that $v_s|_{S_0}$ is
strictly positive; and since $L_{s,S_0}(v_s|_{S_0}) = \rmf(L_s) (v_s|_{S_0})$,
the following lemma implies that $\rmf(L_{s,S_0}) = \rmf(L_s)$. Although this
is a special case of another well-known result, a proof is provided for the
readers' convenience.
\begin{lem}
\label{lem:4.3}
Let $S_0$ be a compact metric space, $\W:= C(S_0)$ and $P=\{f \in
\W: f(t) \ge 0, \, \forall t \in S_0\}$. Assume that $L:\W \to \W$ is
a bounded linear operator such that $L(P) \subset P$. If $w \in \W$
and $w(t) >0$ for all $t \in S_0$, then
\begin{equation*}
\rmf(L)= \lim_{k \rightarrow \infty} \|L^k\|^{1/k} = \lim_{k \rightarrow \infty}
\|L^k w\|^{1/k}.
\end{equation*}
If, in addition, $L w = \lambda w$, then $\lambda = \rmf(L)$; and there exists a
constant $C \ge 1$ such that $\|L^k\| \le C \lambda^k$ for all positive
integers $k$.
\end{lem}
\begin{proof}
Since $S_0$ is compact, there exists $\alpha >0$ such that $w(t) \ge \alpha$ for
all $t \in S_0$. If $f \in \W$ and $\|f\| \le 1$, it follows that for all
$t \in S_0$,
\begin{equation*}
-(1/\alpha) w(t) \le f(t) \le (1/\alpha) w(t).
\end{equation*}
Because $L$ is order-preserving in the partial ordering from $P$,
\begin{equation*}
-(1/\alpha) (L^kw)(t) \le (L^kf)(t) \le (1/\alpha) (L^kw)(t)
\end{equation*}
for all positive integers $k$, which implies that
\begin{equation*}
\|L^k\|^{1/k} = \Big(\sup \{\|L^kf\|: f \in \W, \|f\|\le 1\}\Big)^{1/k} 
\le (1/\alpha)^{1/k} \|L^k w\|^{1/k}.
\end{equation*}
We also have that 
\begin{equation*}
(1/\alpha)^{1/k} \|L^k w\|^{1/k} \le (1/\alpha)^{1/k} \|w\|^{1/k} \|L^k\|^{1/k}.
\end{equation*}
Since
$\lim_{k \rightarrow \infty} \|L^k\|^{1/k} = \rmf(L)$ and
$\lim_{k \rightarrow \infty} (1/\alpha)^{1/k} = \lim_{k \rightarrow \infty}
\|w\|^{1/k} = 1$,
we conclude that $\lim_{k \rightarrow \infty}\|L^k w\|^{1/k} = \rmf(L)$.

If $L w  = \lambda w$, the above argument shows that
\begin{equation*}
(1/\alpha) \lambda^k w(t) \le (L^k f)(t) \le (1/\alpha) \lambda^k w(t),
\end{equation*}
which implies $\|L^k\| \le \tfrac{1}{\alpha} \|w\| \lambda^k$, so the lemma
is satisfied with $C = \tfrac{1}{\alpha} \|w\| \ge 1$.
\end{proof}

Let $S_0=[\amf_0,\bmf_0]$ be a compact interval of reals, $\amf_0 < \bmf_0$,
and let $\B$ be a finite set of real numbers.  For each $\beta \in \B$,
$\theta_{\beta}: S_0 \to \R$. We make the following hypothesis:

(H4) (i) For each $\beta \in \B$, $\theta_{\beta}: S_0 \to S_0$ and 
$\theta_{\beta}$ is a continuous map.

(ii) There exists $\gamma \in \B$ and $\Gamma \in \B$ such that for
all $x \in S_0$ and all $\beta \in \B$, $\theta_{\Gamma}(x) \le \theta_{\beta}(x)
\le \theta_{\gamma}(x)$.

(iii) $x \mapsto \theta_{\gamma}(x)$ and  $x \mapsto \theta_{\Gamma}(x)$ are
strictly decreasing functions on $S_0$.

The example we have in mind is that $\B$ is a finite set of distinct real
numbers with $\beta \ge \gamma >0$ for all $\beta \in \B$ and $\theta_{\beta}(x)
= 1/(x + \beta)$ and $S_0 = [0,1/\gamma]$, but there seems no gain in
specializing at this point.

\begin{lem}
\label{lem:5.1}
Assume (H4). Define $\amf_1 = \theta_{\Gamma}(\bmf_0)$ and $\bmf_1 =
\theta_{\gamma}(\amf_0)$. Then $\amf_0 \le \amf_1 \le \bmf_1 \le
\bmf_0$, $\amf_1 < \bmf_1$, and $\theta_{\beta}(x) \in
[\amf_1,\bmf_1]$ for all $x \in S_0$ and all $\beta \in \B$.
\end{lem}
\begin{proof}
Property (i) in (H4) implies that $\amf_0 \le \amf_1 \le \bmf_0$ and 
$\amf_0 \le \bmf_1 \le \bmf_0$.  Property (ii) implies that $\theta_{\Gamma}(\bmf_0)
= \amf_1 \le \theta_{\gamma}(\bmf_0)$ and  Property (iii) implies that
$\theta_{\gamma}(\bmf_0) <  \theta_{\gamma}(\amf_0) = \bmf_1$, so $\amf_1 < \bmf_1$.
For all $x \in [\amf_0,\bmf_0]$ and all $\beta \in \B$, $\theta_{\Gamma}(x)
\le \theta_{\beta}(x)$ (Property (ii)) and $\theta_{\Gamma}(\bmf_0) 
\le \theta_{\Gamma}(x)$ (Property (iii)), so $\amf_1 \le \theta_{\beta}(x)$.
Similarly, $\theta_{\beta}(x)  \le \theta_{\gamma}(x)$ 
and $\theta_{\gamma}(x) \le \theta_{\gamma}(\amf_0) = \bmf_1$,
so $\theta_{\beta}(x)  \le \bmf_1$.
\end{proof}

\begin{lem}
\label{lem:5.2}
Assume (H4). Also assume that for $1 \le j \le k$, we have found an increasing
sequence of reals $\amf_0 \le \amf_1 \le \ldots \le \amf_k$ and a decreasing
sequence of reals $\bmf_0 \ge \bmf_1 \ge \ldots \ge \bmf_k$ such that $\bmf_j -\amf_j >0$ for
$1 \le j \le k$ and $\theta_{\beta}([\amf_j,\bmf_j]) \subset [\amf_{j+1},\bmf_{j+1}]$ for
$0 \le j \le k-1$  and all $\beta \in \B$. Define $\amf_{k+1} = \theta_{\Gamma}(\bmf_k)$
and $\bmf_{k+1} = \theta_{\gamma}(\amf_k)$. Then we have
$\amf_k \le \amf_{k+1}$, $\bmf_{k+1} \le \bmf_{k}$, $\bmf_{k+1} - \amf_{k+1} >0$ and
$\theta_{\beta}([\amf_k,\bmf_k]) \subset [\amf_{k+1},\bmf_{k+1}]$ for all $\beta \in B$.
\end{lem}
\begin{proof}
Apply Lemma~\ref{lem:5.1} with $[\amf_k,\bmf_k]$ taking the place of $[\amf_0,\bmf_0]$ and
$\theta_{\Gamma}(\bmf_k) = \amf_{k+1}$ taking the place of $\amf_1$ and
$\theta_{\gamma}(\amf_k) = \bmf_{k+1}$ taking the place of $\bmf_1$.
\end{proof}

It follows from Lemma~\ref{lem:5.2} that if (H4) holds and if we
inductively define sequences $\amf_k$ and $\bmf_k$ by
$\amf_{k+1} = \theta_{\Gamma}(\bmf_k)$ and $\bmf_{k+1} = \theta_{\gamma}(\amf_k)$ for $k \ge 0$,
then for all integers $k \ge 0$, $\amf_k < \bmf_k$, $\amf_{k+1} \ge \amf_k$, $\bmf_{k+1} \le \bmf_k$,
and $\theta_{\beta}([\amf_k,\bmf_k]) \subset [\amf_{k+1},\bmf_{k+1}]$ for all $\beta \in \B$.
It follows that $\lim_{k \rightarrow \infty} \amf_k:= \amf_{\infty}$ and
$\lim_{k \rightarrow \infty} \bmf_k:= \bmf_{\infty}$ both exist.

\begin{lem}
\label{lem:5.3}
Assume (H4) and let notation be as above.  Then
$\theta_{\beta}([\amf_{\infty},\bmf_{\infty}]) \subset
[\amf_{\infty},\bmf_{\infty}]$ for all $\beta \in \B$ and
$\theta_{\gamma}(\amf_{\infty}) = \bmf_{\infty}$ and
$\theta_{\Gamma}(\bmf_{\infty}) = \amf_{\infty}$, so $\theta_{\Gamma} \circ
\theta_{\gamma}(\amf_{\infty}) = \amf_{\infty}$ and $\theta_{\gamma} \circ
\theta_{\Gamma}(\bmf_{\infty}) = \bmf_{\infty}$.  If $\beta_1, \beta_2,
\ldots, \beta_k$ are elements of $\B$ and $x \in [\amf_0,\bmf_0]$, then
$(\theta_{\beta_1} \circ \theta_{\beta_2} \circ \cdots \circ
\theta_{\beta_k})(x) \in [\amf_k, \bmf_k]$. If $(\theta_{\beta_1} \circ
\theta_{\beta_2} \circ \cdots \circ \theta_{\beta_k})(x) = x$ for some
$x \in [\amf_0,\bmf_0]$ and some elements $\beta_1, \beta_2, \ldots,
\beta_k$ of $\B$, then $x \in [\amf_{\infty}, \bmf_{\infty}]$.
\end{lem}
\begin{proof}
Because $\lim_{k \rightarrow \infty} \amf_k:= \amf_{\infty}$,
$\lim_{k \rightarrow \infty} \bmf_k:= \bmf_{\infty}$, $\theta_{\gamma}(\amf_k) = \bmf_{k+1}$, and
$\theta_{\Gamma}(\bmf_k) = \amf_{k+1}$, it follows from the continuity of
$\theta_{\gamma}$ and $\theta_{\Gamma}$ that $\theta_{\gamma}(\amf_{\infty}) = \bmf_{\infty}$
and $\theta_{\Gamma}(\bmf_{\infty}) = \amf_{\infty}$.

If $x \in [\amf_0,\bmf_0]$ and $\beta_1, \beta_2, \ldots,
\beta_k$ are elements of $\B$, repeated applications of Lemma~\ref{lem:5.1}
show that $\theta_{\beta_k}(x) \in [\amf_1,\bmf_1]$, $(\theta_{\beta_{k-1}} \circ 
\theta_{\beta_k})(x) \in [\amf_2,\bmf_2]$, and generally that
$(\theta_{\beta_{1}} \circ  \theta_{\beta_{2}} \circ \cdots \circ
\theta_{\beta_k})(x) \in [\amf_k,\bmf_k]$. If  $x \in [\amf_0,\bmf_0]$ and 
$\beta_1, \beta_2, \ldots, \beta_k$ are elements of $\B$ are such that
$(\theta_{\beta_{1}} \circ  \theta_{\beta_{2}} \circ \cdots \circ
\theta_{\beta_k})(x) = x$, it follows that $x \in [\amf_k,\bmf_k]$.  Now the same
argument can be repeated to show that $x \in [\amf_{2k},\bmf_{2k}]$ and
generally that  $x \in [\amf_{mk},\bmf_{mk}]$ for every positive integer $m$. Since
$\cap_{m \ge 1}[\amf_{mk},\bmf_{mk}] = [\amf_{\infty}, \bmf_{\infty}]$, we conclude that
$x \in [\amf_{\infty},\bmf_{\infty}]$.
\end{proof}

\begin{remark}
Under the hypotheses of Lemma~\ref{lem:5.3}, if 
$(\theta_{\Gamma} \circ \theta_{\gamma})(x) = x$ or
$(\theta_{\gamma} \circ \theta_{\Gamma})(x) = x$ for some $x \in [\amf_0,\bmf_0]$, then
$\amf_{\infty} \le x \le \bmf_{\infty}$, so $\amf_{\infty}$ is the least fixed point
of $\theta_{\Gamma} \circ \theta_{\gamma}$ in $[\amf_0,\bmf_0]$ and 
$\bmf_{\infty}$ is the greatest fixed point
of $\theta_{\gamma} \circ \theta_{\Gamma}$ in $[\amf_0,\bmf_0]$ .
\end{remark}

Lemma~\ref{lem:5.3} provides a way of obtaining invariant intervals
$J$, such that $\theta_{\beta}(J) \subset J$ for all $\beta \in \B$.
However, it is frequently the case that we have more information than
given in (H4), and then one can give more flexible methods to find
invariant intervals.  The following lemma, whose proof we omit,
describes a commonly occurring class of examples where such methods
are available.
\begin{lem}
\label{lem:7.3'}
Let hypotheses and notation be as in Lemma~\ref{lem:5.3}.  Suppose also
that there exist intervals $J_1 =[x_1, \amf_{\infty}]$ and $J_2 =
[\bmf_{\infty},x_2]$
with $\amf_0 \le x_1 < \amf_{\infty}$ and $\bmf_{\infty} < x_2 < \bmf_0$, 
such that
$\theta_{\gamma}(J_2) \subset [\amf_{\infty}, \bmf_{\infty}]$, 
$\theta_{\Gamma}(J_1) \subset [\amf_{\infty}, \bmf_{\infty}]$, 
$\lip(\theta_{\gamma}|_{J_1}) \le c_1$, $\lip(\theta_{\Gamma}|_{J_2}) \le
c_2$, and $c_1 c_2 <1$.  If $\xi_1 \in J_1$ is chosen so that $\xi_2 =
\theta_{\gamma}(\xi_1) \in J_2$, then
$\theta_{\beta}([\xi_1,\xi_2]) \subset [\xi_1,\xi_2]$ for all $\beta \in \B$.
Similarly, if $\eta_2 \in J_2$ is chosen so that $\eta_1 =
\theta_{\Gamma}(\eta_2) \in J_1$, then
$\theta_{\beta}([\eta_1,\eta_2]) \subset [\eta_1,\eta_2]$ for all $\beta \in \B$.
\end{lem}

We shall use $\B$ as an index set, so the operator $L_s$ can be written
\begin{equation*}
(L_s f)(x) = \sum_{\beta \in \B} g_{\beta}(x)^s f(\theta_{\beta}(x)),
\end{equation*}
where $\theta_{\beta}(S_0) \subset S_0$ for all $\beta \in \B$ and $S_0 =
[\amf_0,\bmf_0]$. If the conditions of Theorem~\ref{thm:2.4} are satisfied,
$L_s$ has a strictly positive, $C^m$ eigenfunction. Assuming (H4) and the
hypotheses of Theorem~\ref{thm:2.4}, the observation in the first paragraph of
this section implies that to compute $\rmf(L_s)$, we can, in the notation of
Lemma~\ref{lem:5.3}, replace $[\amf_0,\bmf_0]$ by
$[\amf_{\infty},\bmf_{\infty}]$ or by $[\amf_k,\bmf_k]$ for any integer $k \ge
1$. In fact, we could use any interval $J \subset [\amf_0,\bmf_0]$ with
$\theta_{\beta}(J) \subset J$ for all $\beta \in \B$ (compare
Lemma~\ref{lem:7.3'}).

For the remainder of this section, we shall assume

(H5) $\B$ is a finite set of distinct real numbers and
$\gamma= \min \{\beta: \beta \in \B\} \ge 1$.  For every $\beta \in \B$, we
define $\theta_{\beta}:[0,1/\gamma]:=[\amf_0,\bmf_0] \to [0,1/\gamma]$ 
by $\theta_{\beta}(x)  = (x+ \beta)^{-1}$.

We shall write $\Gamma = \max \{\beta: \beta \in \B\}$ and $\gamma =
\min \{\beta: \beta \in \B\}$ and always assume that $\gamma <
\Gamma$.  The reader can check that $\{\theta_{\beta}: \beta \in \B\}$
satisfies the conditions of (H4) with $\theta_{\gamma}(x) = (x +
\gamma)^{-1}$ and $\theta_{\Gamma}(x) = (x + \Gamma)^{-1}$. Using the
calculations in the following paragraph, the reader can check that the
conditions of Lemma~\ref{lem:7.3'} are also satisfied.

We assume that the sequences $\{\amf_k: k \ge 1\}$ and $\{\bmf_k: k \ge 1\}$
are defined as in Lemmas~\ref{lem:5.2} and \ref{lem:5.3}, with $\amf_0 =0$ and
$\bmf_0 =1/\gamma$, and $\amf_{\infty}$ and $\bmf_{\infty}$ defined as in
Lemma~\ref{lem:5.3}.  Since $\amf_{\infty}$ is a fixed point of
$\theta_{\Gamma} \circ \theta_{\gamma}$ in $[0,1/\gamma]$ and $\bmf_{\infty}$ is a
fixed point of $\theta_{\gamma} \circ \theta_{\Gamma}$ in $[0,1/\gamma]$, one can
easily solve the equations
\begin{equation*}
 x= (\theta_{\Gamma} \circ \theta_{\gamma})(x) = \frac{x + \gamma}{\Gamma x +
   (1+ \Gamma \gamma)} \quad \text{and}
\quad x = (\theta_{\gamma} \circ \theta_{\Gamma})(x)
= \frac{x + \Gamma}{\gamma x +    (1+ \Gamma \gamma)}
\end{equation*}
to obtain
\begin{equation}
\label{5.1}
\amf_{\infty} = - \frac{\gamma}{2} + \sqrt{(\gamma/2)^2 + (\gamma/\Gamma)}
\quad \text{and} \quad
\bmf_{\infty} = - \frac{\Gamma}{2} + \sqrt{(\Gamma/2)^2 + (\Gamma/\gamma)}.
\end{equation}
One can verify that $\bmf_{\infty} = (\Gamma/\gamma) \amf_{\infty}$, so
$0 < \amf_{\infty} < \bmf_{\infty} <1/\gamma$.

Since our index set is $\B$, we slightly abuse previous notation and, for
a positive integer $\nu$, we define the set of ordered $\nu$-tuples of
elements of $\B$ by
\begin{equation*}
\Omega_{\nu} = \{(\beta_1, \beta_2, \cdots, \beta_{\nu}): \beta_j \in \B
\text{ for } 1 \le j \le \nu\}.
\end{equation*}
For each $\omega = (\beta_1, \beta_2, \cdots, \beta_{\nu}) \in \Omega_{\nu}$,
we define $\theta_{\omega} = \theta_{\beta_1} \circ \theta_{\beta_2} \circ
\cdots \circ \theta_{\beta_{\nu}}$. Our first task is to estimate $c(\nu)$
(see \eqref{3.4}), which gives an upper bound for $\lip(\theta_{\omega})$,
$\omega \in \Omega_{\nu}$.

If $\omega = (\beta_1, \beta_2, \cdots, \beta_{\nu}) \in \Omega_{\nu}$ and 
$\beta \in \B$, define a matrix
\begin{equation*}
M_{\beta} = \begin{pmatrix} 0 & 1 \\ 1 & \beta \end{pmatrix}.
\end{equation*}
It is proved in Section 6 of \cite{hdcomp1} that
\begin{equation*}
M = M_{\beta_1} M_{\beta_2} \cdots M_{\beta_{\nu}} 
= \begin{pmatrix} A_{\nu-1} & A_{\nu} \\ B_{\nu-1} & B_{\nu} \end{pmatrix},
\end{equation*}
where $A_j$ and $B_j$ are defined inductively by $A_0 =0$, $A_1 =1$,
$B_0 =1 $, $B_1 = \beta_1$, and generally, for $1 \le j \le \nu$, by
\begin{equation}
\label{5.2}
A_{j+1} = A_{j-1} + \beta_{j+1} A_j, \qquad 
B_{j+1} = B_{j-1} + \beta_{j+1} B_j.
\end{equation}
Note that $\det(M_{\beta}) = -1$ so $\det(M) = (-1)^{\nu}$.  Standard results
for M\"obius transforms now imply that for $x \in [\amf_k,\bmf_k]$, 
$0 \le k \le \infty$,
\begin{align}
\nonumber 
(\theta_{\beta_1} \circ \theta_{\beta_2} \circ \cdots \circ
\theta_{\beta_{\nu}})(x)  &= \frac{A_{\nu-1} x + A_{\nu}}{B_{\nu-1} x +
  B_{\nu}},
\\
\label{thetabetaest2}
\frac{d}{dx}(\theta_{\beta_1} \circ \theta_{\beta_2} \circ \cdots \circ
\theta_{\beta_{\nu}})(x)  &= \frac{(-1)^{\nu}}{(B_{\nu-1} x +  B_{\nu})^2}.
\end{align}
If we define $\tilde B_0 = 1$, $\tilde B_1 = \gamma$, and
$\tilde B_{j+1} = \tilde B_{j-1} + \gamma \tilde B_j$ for $j \ge 1$, then
because $\gamma \le \beta$ for all $\beta \in \B$, it is straightforward
to prove that $\tilde B_j \le B_j$ for $0 \le j \le \nu$, where $B_j$ is
defined by \eqref{5.2}. It follows that for all $\omega \in \Omega_{\nu}$
and $x \in [\amf_k, \bmf_k]$,
\begin{equation*}
|\theta_{\omega}^{\prime}(x)| \le [\tilde B_{\nu-1} \amf_k + \tilde
B_{\nu}]^{-2},
\end{equation*}
which implies that, for $\theta_{\omega}:[\amf_k, \bmf_k] \to \R$,
\begin{equation}
\label{5.4}
\max \{\lip(\theta_{\omega}): \omega \in \Omega_{\nu}\}
= [\tilde B_{\nu-1} \amf_k + \tilde B_{\nu}]^{-2}:= c(\nu).
\end{equation}

It remains to give an exact formula for the right hand side of \eqref{5.4}.
The linear difference equation 
$\tilde B_{j+1} = \tilde B_{j-1} + \gamma \tilde B_j$  has solutions of
the form $\lambda^j$ for $j \ge 0$, which leads to the formula
$\lambda^{n+1} = \lambda^{n-1} + \gamma \lambda^n$, or for $\lambda \neq 0$,
$\lambda^2 = 1 + \lambda \gamma$. Hence
\begin{equation}
\label{5.5}
\lambda = \lambda_{+} = \frac{\gamma}{2} + \frac{1}{2}\sqrt{\gamma^2 +4},
\qquad
\lambda = \lambda_{-} = \frac{\gamma}{2} - \frac{1}{2}\sqrt{\gamma^2 +4}.
\end{equation}
The general solution of the difference equation is then
\begin{equation}
\label{5.6}
c_1 \lambda_{+}^j + c_2 \lambda_{-}^j = \tilde B_j, \quad j \ge 0,
\end{equation}
where $c_1$ and $c_2$ must be chosen so that $\tilde B_0 =1$ and 
$\tilde B_1 = \gamma$. A calculation gives
\begin{equation}
\label{5.7}
c_1 = \frac{\sqrt{\gamma^2 + 4} + \gamma}{2\sqrt{\gamma^2 + 4}},
\qquad
c_2 = \frac{\sqrt{\gamma^2 + 4} - \gamma}{2\sqrt{\gamma^2 + 4}}.
\end{equation}

Summarizing the above discussion, we obtain
\begin{lem}
\label{lem:5.4}
Assume (H4) and consider $\theta_{\beta}$, $\beta \in \B$, as a map of
$[\amf_k,\bmf_k]$ to itself, for $0 \le k \le \infty$, where $\amf_0 =0$,
$\bmf_0 =1$, $\amf_k = \theta_{\Gamma}(\bmf_{k-1})$ and $\bmf_k =
\theta_{\gamma}(\amf_{k-1})$ for $k \ge 1$ and $\amf_{\infty}$ and
$\bmf_{\infty}$ are given by \eqref{5.1}. Then for $j \ge 1$, $\tilde B_j$ is
given by \eqref{5.6}, where $\lambda_{+}$ and $\lambda_{-}$ are given by
\eqref{5.5} and $c_1$ and $c_2$ by \eqref{5.7}.
\end{lem}

\begin{remark}
Because $\lambda_{+} >1$ and
$-1 < \lambda_{-} = - 1/\lambda_{+} <0$ for all $\gamma >0$,
$c_1 \lambda_{+}^j$ is the dominant term in  \eqref{5.6} as $j$ increases; and
one can check that $|c_2 \lambda_{-}^j| < 1/2$ for all $j \ge 0$.  Of course,
for moderate values of $j$, one can easily compute $\tilde B_j$  from its
recurrence formula.  It is clear that the constant $c(\nu)$ in 
\eqref{5.4} is minimized by working on the interval $[\amf_{\infty},
\bmf_{\infty}]$.
\end{remark}

Assuming (H5), we now define for $s >0$, $L_s:C([0,1]) \to C([0,1])$ by
\begin{equation*}
(L_sf)(x) = \sum_{\beta \in \B} |\theta_{\beta}^{\prime}(x)|^s 
f(\theta_{\beta}(x)):= 
\sum_{\beta \in \B} g_{\beta}(x)^s  f(\theta_{\beta}(x)).
\end{equation*}
It is well-known that for $\nu$ a positive integer,
\begin{equation*}
(L_s^{\nu} f)(x) = \sum_{\beta \in \B} |\theta_{\omega}^{\prime}(x)|^s 
f(\theta_{\omega}(x)):= 
\sum_{\omega \in \Omega_{\nu}} g_{\omega}(x)^s  f(\theta_{\omega}(x)).
\end{equation*}
Because $\theta_{\omega}([\amf_k,\bmf_k]) \subset [\amf_k,\bmf_k]$ for $0 \le k \le \infty$
and $\omega \in \Omega_{\nu}$, we can also consider $L_s^{\nu}$ as a map
of $C([\amf_k,\bmf_k]) \mapsto C([\amf_k,\bmf_k])$ and as noted earlier, this does
not change the spectral radius of $L_s^{\nu}$.  Thus, we shall consider
$L_s^{\nu}$ as a map from $C([\amf_k,\bmf_k])$ into itself, with optimal results
obtained by taking $k = \infty$.

We need to find a constant $M_0(\nu)$ (compare \eqref{3.4}) such that
for all $\omega \in \Omega_{\nu}$, $g_{\omega}(x): =
|\theta_{\omega}^{\prime}(x)| \in K(M_0(\nu); [\amf_k,\bmf_k])$. In
this case, this is equivalent to proving that for all $\omega \in
\Omega_{\nu}$, $x \mapsto \ln (|\theta_{\omega}^{\prime}(x)|)$ is a
Lipschitz map on $[\amf_k,\bmf_k]$ with Lipschitz constant $\le
M_0(\nu)$. If $\omega = (\beta_1, \beta_2, \ldots, \beta_{\nu})$, we
have by \eqref{thetabetaest2}, that
\begin{equation*}
|\theta_{\omega}^{\prime}(x)| = \frac{1}{(B_{\nu-1} x + B_{\nu})^2},
\end{equation*}
so
\begin{equation*}
\ln(|\theta_{\omega}^{\prime}(x)|) = -2 \ln(B_{\nu-1} x + B_{\nu}).
\end{equation*}
Thus it suffices to choose $M_0(\nu)$ so that for all $x \in [\amf_k,\bmf_k]$ and
all $\omega \in \Omega_{\nu}$,
\begin{multline*}
\Big|\frac{d}{dx}\ln(|\theta_{\omega}^{\prime}(x)|) \Big|
= 2 \frac{B_{\nu-1}}{B_{\nu -1} x + B_{\nu}}
\\
= 2 \Big[\frac{1}{x + (B_{\nu}/B_{\nu-1})}\Big]
\le 2 \Big[\frac{1}{\amf_k + (B_{\nu}/B_{\nu-1})}\Big] \le M_0(\nu).
\end{multline*}

If we define  $x_j = B_{j-1}/B_j$ for $1  \le j < \nu$, then  since $B_{j+1} =
B_{j-1}  + \beta_{j+1}  B_j$ for  $1 \le  j \le  \nu$, we  get $B_{j+1}/B_j  =
B_{j-1}/B_j   +  \beta_{j+1}$   or  $x_{j+1}   =  1/(x_j   +  \beta_{j+1})   =
\theta_{\beta_j +1}(x_j)$  for $1 \le j  \le \nu$. Since $x_1  = 1/\beta_1 \in
[\amf_1,\bmf_1]  =   [1/(\Gamma  +1/\gamma),   1/\gamma]$,  it   follows  from
Lemma~\ref{5.2} that  $x_{j+1} \in  [\amf_{j+1}, \bmf_{j+1}]$ for  $1 \le  j <
\nu$, so $1/x_{j+1} =  B_{j+1}/B_j \in [\bmf_{j+1}^{-1}, \amf_{j+1}^{-1}]$ and
$\B_{\nu}/B_{\nu-1}  \ge \bmf_{\nu}^{-1}$.   It follows  that for  $\omega \in
\Omega_{\nu}$ and $x \in [\amf_k,\bmf_k]$, we can take 
\begin{equation}
\label{M0nuchoice}
M_0(\nu) = 2/(\amf_k + \bmf_{\nu}^{-1}).
\end{equation}
By  adding the exponent  $s$, one easily derives  that for
all $\omega \in \Omega_{\nu}$, $\nu \ge 1$, and $0 \le k \le \infty$,
\begin{equation}
\label{5.10a}
g_{\omega}(\cdot)^s = |\theta_{\omega}^{\prime}(\cdot)|^s \in K(2s/(\amf_k +
  \bmf_{\nu}^{-1}); [\amf_k,\bmf_k]).
\end{equation}
Note that we could replace $\bmf_{\nu}^{-1}$ by $\amf_{\nu-1} + \gamma$.

We summarize the above discussion in the following lemma.
\begin{lem}
\label{lem:5.5}
Assume (H5) and let $\amf_k$ and $\bmf_k$, $0 \le k \le \infty$, be as described
in Lemma~\ref{lem:5.4}. If $\omega \in \Omega_{\nu}$, $\nu \ge 1$, the map
$x \mapsto \ln(|\theta_{\omega}^{\prime}(x)|^s)$, $x \in [\amf_k,\bmf_k]$ is
Lipschitz with Lipschitz constant $\le 2s/(\amf_k +   \bmf_{\nu}^{-1})$,
so \eqref{5.10a} is satisfied.
\end{lem}

It remains to estimate, for $0 \le k \le \infty$,
$\max_{x \in [\amf_k,\bmf_k]} |(d^j v_s/dx^j)(x)|/v_s(x)$, where $v_s$ denotes the unique
(to within normalization) strictly positive eigenfunction of $L_s$.  The basic
idea is to exploit \eqref{2.15}, as was done in Section 6 of \cite{hdcomp1},
but our results will refine those in \cite{hdcomp1}.

Our previous calculations show that for $x \in [\amf_k,\bmf_k]$, $0 \le k \le
\infty$, we have
\begin{equation*}
g_{\omega}(x)^s = |\theta_{\omega}^{\prime}(x)|^s 
= \frac{1}{B_{\nu -1}^{2s}( x + B_{\nu}/B_{\nu-1})^{2s}}.
\end{equation*}
It follows that for $j \ge 1$, and letting $D$ denote $d/dx$, we have
\begin{equation}
\label{5.10}
\frac{(-1)^j (D^j[(g_{\omega})^s])(x)}{g_{\omega}(x)}
= \frac{(2s)(2s+1) \cdots (2s+j-1)}{( x + B_{\nu}/B_{\nu-1})^j}.
\end{equation}
The same argument used in Lemma~\ref{lem:5.5} shows that
\begin{equation}
\label{5.11}
\bmf_{\nu}^{-1} \le B_{\nu}/B_{\nu-1} \le \amf_{\nu}^{-1},
\end{equation}
so if  $x \in [\amf_k,\bmf_k]$, we derive from \eqref{5.10} and \eqref{5.11} that
\begin{multline}
\label{5.12}
\frac{(2s)(2s+1) \cdots (2s+j-1)}{[\bmf_k + \amf_{\nu}^{-1}]^j}
\le \frac{(-1)^j (D^j[(g_{\omega})^s])(x)}{g_{\omega}(x)^s}
\\
\le \frac{(2s)(2s+1) \cdots (2s+j-1)}{(\amf_k + \bmf_{\nu}^{-1})^j}.
\end{multline}
It follows from \eqref{5.12} that for $x \in [\amf_k,\bmf_k]$,
\begin{multline}
\label{5.13}
\frac{(2s)(2s+1) \cdots (2s+j-1)}{[\bmf_k + \amf_{\nu}^{-1}]^j}
\le \frac{(-1)^j \sum_{\omega \in \Omega_{\nu}} (D^j[(g_{\omega}^s])(x)}
{\sum_{\omega \in \Omega_{\nu}} g_{\omega}(x)^s}
\\
\le \frac{(2s)(2s+1) \cdots (2s+j-1)}{(\amf_k + \bmf_{\nu}^{-1})^j}.
\end{multline}
Taking limits as $\nu \rightarrow \infty$ in \eqref{5.13} and using
\eqref{2.15}, we find that for $x \in [\amf_k,\bmf_k]$,
\begin{multline}
\label{5.14}
\frac{(2s)(2s+1) \cdots (2s+j-1)}{[\bmf_k + \amf_{\infty}^{-1}]^j}
\le \frac{(-1)^j (d^j v_s/dx^j)(x)}{v_s(x)}
\\
\le \frac{(2s)(2s+1) \cdots (2s+j-1)}{(\amf_k + \bmf_{\infty}^{-1})^j}.
\end{multline}
Notice that we can replace $\amf_{\infty}^{-1}$ by $\bmf_{\infty} + \Gamma$ and
$\bmf_{\infty}^{-1}$ by $\amf_{\infty} + \gamma$ in \eqref{5.14}.

As one can easily see, the lower bound in \eqref{5.14} increases as $k$
increases and the upper bound decreases as $k$ increases, so the optimal
bounds are obtained when $k = \infty$ and apply to the interval 
$[\amf_{\infty}, \bmf_{\infty}]$.

We summarize the above results in the following
lemma.
\begin{lem}
\label{lem:5.6} 
Let $v_s$ denote the unique
(to with normalization) strictly positive eigenfunction of $L_s$. Assume (H4)
and let $\amf_k$ and $\bmf_k$, $k \ge 0$ be as in Lemma~\ref{5.4}. Then $v_s$
satisfies \eqref{5.14}.
\end{lem}

\begin{remark}
\label{rem:7.3} 
Since, in Lemma~\ref{lem:5.6}, we have specified the coefficient $g_{\beta}$
  and the maps $\theta_{\beta}$ for $\beta \in \B$, Lemma~\ref{lem:5.6} gives
  us a simple formula for the constant $E(s,p) = E$ in \eqref{6.1}.
\begin{equation*}
\max_{x \in [\amf_k,\bmf_k]} \frac{|(d^p v_s/dx^p)(x)|}{v_s(x)}
\le \frac{(2s)(2s+1) \cdots (2s+p-1)}{(\amf_k + \bmf_{\infty}^{-1})^p} := E,
\end{equation*}
where $p$ and $k$ are positive integers and $ 1 \le k \le \infty$.  Here
we have allowed the interval to vary with $k$, but we may eventually
restrict to $k = \infty$.
\end{remark}

It remains to find a constant $\chi$ (compare \eqref{6.2}) such that
for all $x_1, x_2 \in [\amf_k,\bmf_k]$,
\begin{equation*}
v_s(x_1) \le \exp(2 s|x_1 -x_2|/\chi) v_s(x_2).
\end{equation*}
It follows from \eqref{5.14} that if $\amf_k \le x_1 \le x_2 \le \bmf_k$, then
\begin{equation*}
- \int_{x_1}^{x_2} \frac{v_s^{\prime}(x)}{v_s(x)} \, dx
= \ln(v_s(x_1)) - \ln(v_s(x_2)) \le \frac{2s}{\amf_k + \bmf_{\infty}^{-1}}
|x_2 - x_1|,
\end{equation*}
which implies that 
\begin{equation}
\label{5.17}
v_s(x_1) \le \exp(2 s|x_1 -x_2|/[\amf_k + \bmf_{\infty}^{-1}]) v_s(x_2).
\end{equation}

If $x_2 \le x_1$, we know that $v_s(x_2) \ge v_s(x_1)$, so \eqref{5.17}
is satisfied for all $x_1,x_2 \in [\amf_k,\bmf_k]$. In particular, \eqref{5.17}
is satisfied if the roles of $x_1$ and $x_2$ are reversed, which implies
that $x \mapsto \ln(v_s(x))$ is Lipschitz on $[\amf_k,\bmf_k]$ with
Lipschitz constant $2s/(\amf_k + \bmf_{\infty}^{-1})$.  Summarizing, we have
\begin{lem}
\label{lem:5.7}
If $\chi = \amf_k + \bmf_{\infty}^{-1}$, \eqref{6.2} is satisfied on $[\amf_k,\bmf_k]$.
\end{lem}

\section{Computation of $\rmf(L_s)$}
\label{sec:computation}

In this section we shall describe how to to use the results of
Sections~\ref{sec:theory-d}-\ref{sec:est-constants} to obtain rigorous, high
order estimates for $\rmf(L_s) = \lambda_s$.  As a subcase, we shall obtain
rigorous estimates for the Hausdorff dimension of certain fractal objects
described by iterated function systems.

For simplicity, we shall restrict attention to the class of maps
$\theta_{\beta}: [0,1] \to [0,1]$, where $\theta_{\beta}(x) = 1/(x + \beta)$
for $\beta \in \B$ and $\B$ as in (H5) of Section~\ref{sec:est-constants}.
For $s \ge 0$, we define $L_s: C[0,1] \to C[0,1]$ by
\begin{equation*}
(L_s f)(x) = \sum_{\beta \in \B} |\theta_{\beta}^{\prime}(x)|^s 
f(\theta_{\beta}(x)).
\end{equation*}

Recall (see Lemmas~\ref{lem:5.2} and \ref{lem:5.3}) that we define $\amf_0
=0$, $\bmf_0 =1/\gamma$, $\amf_{k+1} = \theta_{\Gamma}(\bmf_k)$, $\bmf_{k+1} =
\theta_{\gamma}(\amf_k)$, $\amf_{\infty} = \lim_{k \rightarrow \infty}
\amf_k$, and $\bmf_{\infty} = \lim_{k \rightarrow \infty} \bmf_k$.  Since
$\theta_{\beta}([\amf_k,\bmf_k]) \subset [\amf_k,\bmf_k]$ for $0 \le k \le
\infty$ and for all $\beta \in \B$, and since $L_s$ has a strictly positive
eigenvector on $[0,1/\gamma]$, $L_s$ induces a bounded linear operator
$L_{s,[\amf_k,\bmf_k]} : C([\amf_k,\bmf_k]) \to C([\amf_k,\bmf_k])$ and
$\rmf(L_s) = \rmf(L_{s,[\amf_k,\bmf_k]})$, Various constants are optimized by
working on $[\amf_{\infty}, \bmf_{\infty}]$, so we shall abuse notation and
also use $L_s$ to denote $L_s$ as an operator on $C([\amf_{\infty},
\bmf_{\infty}])$ (or, sometimes, $C([\amf_{k}, \bmf_{k}])$).  For a given
positive integer $r$, we assume that (H3) is satisfied, but with $S:=
[\amf_{\infty}, \bmf_{\infty}]$.

Thus $[a_i, b_i] \subset S$, $1 \le i
\le I$, denote pairwise disjoint intervals that satisfy the conditions of
(H3), Given positive integers $N_i$, $1 \le i \le I$, we write $h_i = (b_i
- a_i)/N_i$.  As in Section~\ref{sec:approx} (see \eqref{2.33}), we define
mesh points $c_{j,k}^i \in [a_i, b_i]$, $1 \le i \le I$, $1 \le j \le
N_i$, $0 \le k \le r$ and $T:=\{c_{j,k}^i\}$ for $i,j,k$ in the ranges given
above.  As in Section~\ref{sec:approx}, if $v_s$ is the positive eigenvector of
$L_s$ on $S$, $\cV_s: \hat S:= \cup_{i=1}^I [a_i, b_i] \to \R$ is the 
polynomial interpolant of $v_s$ of degree $\le r$ on $[t_{j-1}^i, t_j^i]$ for
$1 \le i \le I$, $1 \le j \le N_i$, so $\cV_s(x) = v_s(x)$ for all $x \in T$.

Our general approach will be as follows: Given $s >0$, we must find
$r, \nu, M$ and $h$ such that the conditions of Theorem~\ref{thm:5.6n} are
satisfied. First, we choose a positive integer $r \ge 2$, where $r$ is the
piecewise polynomial degree in \eqref{2.35}.  Once $r$ has been chosen, we
select a positive integer $\nu$ such that (compare Remark~\ref{rem:3.2})
\begin{equation}
\label{8.2}
c(\nu)[2 \eta(r) r^2 \psi(r)]:= \kappa_1 <1.
\end{equation}
Here $c(\nu)$ is as in \eqref{3.4}; and for our case an exact formula
for $c(\nu)$ is provided by \eqref{5.4}; where we shall take $\amf_k =
\amf_{\infty}$ in \eqref{5.4}. Also, $\psi(r)$ is as in
Lemma~\ref{lem:3.2} and $\eta(r)$ as in \eqref{3.1}.  As a practical
matter, we demand that $\kappa_1$ not be too close to 1, say $\kappa_1
\le 4/5$. Note that for fixed $r$, this means that $\nu$
  must be sufficiently large and hence $c(\nu)$ sufficiently small, so
  that \eqref{8.2} is satisfied.  We next choose $\kappa_2$ with
  $\kappa_1 < \kappa_2 <1$ and $\kappa_2$ not too close to 1.  A
  simple choice is $\kappa_2 = (1 + \kappa_1)/2$.  We define (see
  Theorem~\ref{thm:5.6n}), $M^{\prime} = \kappa_2 M$.  If we write
$u:= M \eta(r) h$, the conditions of Theorem~\ref{thm:5.6n} take the
form
\begin{align}
\label{8.3} 
\psi(r) u \exp(u) &< 1
\\
\label{8.4}
\frac{\kappa_1 \exp(u)}{1- \psi(r) u \exp(u)} &< \kappa_2 - \frac{s
  M_0(\nu)}{M}.
\end{align}
Here $M_0(\nu)$ is as in \eqref{3.4}; and in our case Lemma~\ref{lem:5.5}
insures that $M_0(\nu) \le 2/(\amf_{\infty} + \amf_{\nu-1} + \gamma)$.

Since $\exp(u)/(1- \psi(r) u \exp(u)) > 1$, \eqref{8.4} implies that
\begin{equation*}
\kappa_1 < \kappa_2 - sM_0(\nu)/M.
\end{equation*}
We choose $M >0$ such that
\begin{equation}
\label{8.6}
M = \frac{4s}{\amf_{\infty} + \amf_{\nu-1} + \gamma}\frac{1}{\kappa_2- \kappa_1}
\ge 2 sM_0(\nu)/(\kappa_2- \kappa_1),
\end{equation}
which implies that 
\begin{equation*}
\kappa_2 - sM_0(\nu)/M \ge \kappa_2 - (\kappa_2- \kappa_1)/2 > \kappa_1.
\end{equation*}
Also note that since $\amf_{\infty} + \amf_{\nu-1} + \gamma < \chi$, we have that
\begin{equation}
\label{8.6a}
M \ge \frac{4s}{\chi}\frac{1}{\kappa_2- \kappa_1}.
\end{equation}

Given an $M$ that satisfies \eqref{8.6}, we can choose $h = \max_i h_i$
sufficiently small, i.e., $h \le \amh_0$, that \eqref{8.3} and \eqref{8.4} are
satisfied. Recall, however, that we also have to insure that the constant $M$,
defined by \eqref{8.6} also satisfies $M \ge M_2$, where $M_2$ is as in
Lemma~\ref{lem:6.4} and $M_1$ is given by \eqref{6.12}.  As we shall see, this
may require a further restriction on the size of $h$.

The constants $M_1$ and $M_2$ are defined in terms of $G_{r+1}:=G$ (compare
\eqref{rn:6.8}), $\chi:= 2 \amf_{\infty} + \gamma$, $s$, and $H$ (compare
Lemma~\ref{lem:6.4}), and it is desirable to choose $H$ to be small.  The
constant $E$ in $G_{r+1}$ is given by Remark~\ref{rem:7.3} with $p =r+1$.
By using \eqref{6.9} in Section~\ref{sec:est-constants} and the estimates for
$E$ and $\chi$ in Lemmas~\ref{lem:5.6} and \ref{lem:5.7}, we find that we can
write
\begin{equation*}
G = E \exp\Big(\frac{2sh}{\amf_{\infty} + \bmf_{\infty}^{-1}}\Big)
\frac{1}{(r+1)!} \Big[\frac{1}{2 \cos(\pi/(2r+2))}\Big]^{r+1} \frac{1}{2^r},
\end{equation*}
where
\begin{equation}
\label{Erp1}
E:= \frac{(2s)(2s+1) \cdots (2s+r)}{(\amf_{\infty} + \bmf_{\infty}^{-1})^{r+1}}
= \frac{(2s)(2s+1) \cdots (2s+r)}{(2 \amf_{\infty} + \gamma)^{r+1}},
\end{equation}
and we have used the fact that $\bmf_{\infty} = 1/(\amf_{\infty} + \gamma)$.
Note that in the application of \eqref{5.14} for $E$, one must take $j = r+1$.
A calculation gives
\begin{multline}
\label{5.19}
G:= G_{r+1} = 2 \exp\Big(\frac{2sh}{2 \amf_{\infty} + \gamma}\Big)
\Big[\frac{(2s)(2s+1) \cdots (2s+r)}{(2)(4)(6) \cdots (2r+2)}\Big]
\\
\cdot 
\Big[\frac{1}{2 \amf_{\infty} + \gamma}\Big]^{r+1} \Big[\frac{1}{2 \cos(\pi/(2r+2))}
\Big]^{r+1}.
\end{multline}
Finally, we define
\begin{equation}
\label{8.9}
D:= D_{r+1} = \Big[\frac{1}{\sin(\pi/(2r+2))}\Big]^2 G_{r+1}
\end{equation}
and
\begin{equation}
\label{Hr1def}
H_{r+1} = \mu D_{r+1} \frac{\chi}{2s} = 
\mu \Big[\frac{1}{\sin (\pi/(2r+2))}\Big]^2 G_{r+1}
\frac{(2 \amf_{\infty} + \gamma)}{2s}.
\end{equation}
To estimate $M_1$ and $M_2$, we shall need estimates on these quantities.

\begin{lem}
\label{lem:G-decrease}
Assume that $r \ge 2$, $0 < s <2$, and $2 \amf_{\infty} + \gamma \ge 1$.
Then $G_r$ is a decreasing function of $r$.
\end{lem}
\begin{proof}
For $r \ge 2$ and $0 <s \le 2$, since $2 j+2 \ge  2s+j$ for $j \ge 2$,
it follows that
\begin{equation*}
\prod_{j=0}^r \frac{2s+j}{2j+2} \le \Big(\frac{2s}{2}\Big) \Big(\frac{2s+1}{4}
\Big) = s\frac{2s+1}{4}.
\end{equation*}

By using the Taylor series for $\cos(\theta)$, we see that
\begin{equation*}
\frac{1}{2 \cos(\pi/(2r+2))} \le \frac{1}{2 - [\pi/(2r+2)]^2}.
\end{equation*}
Using these estimates in \eqref{5.19} gives
\begin{equation*}
G_{r+1} \le 2 \exp\Big(\frac{2sh}{2 \amf_{\infty} + \gamma}\Big)
\Big[s\frac{2s+1}{4}\Big] \Big[\frac{1}{2 \amf_{\infty} + \gamma}\Big]^{r+1}
\Big[\frac{1}{2 - [\pi/(2r+2)]^2}\Big]^{r+1},
\end{equation*}
which implies that $\lim_{r \rightarrow \infty} G_{r+1} =0$.
Furthermore, for $r \ge 2$ and $2 \amf_{\infty} + \gamma \ge 1$,
another calculation shows that
\begin{equation*}
\frac{G_{r+1}}{G_r} = \Big[\frac{2s+r}{2r+2}\Big] 
\Big[\frac{1}{2 \amf_{\infty} + \gamma}\Big] \Big[\frac{1}{2 \cos(\pi/[2r+2])}\Big]
\Big[\frac{\cos(\pi/(2r))}{\cos(\pi/(2r+2))}\Big]^r< 1,
\end{equation*}
so $G_r$ is a decreasing function for integer $r \ge 2$.
\end{proof}

If we set
\begin{equation*}
u_{r+1} = \frac{1}{\chi} \Big[\frac{1}{2 \cos(\pi/(2r+2))}\Big],
\end{equation*}
a calculation gives
\begin{equation*}
u_3 = \frac{1}{\chi \sqrt{3}}, \qquad
u_4 = \frac{1}{\chi \sqrt{2 + \sqrt{2}}},
\end{equation*}
which gives
\begin{align}
\label{8.10}
G_3 &= 2s \Big[\frac{1}{\chi\sqrt{3}}\Big]^3\Big[\frac{2s+1}{4} \Big]
\Big[\frac{s+1}{3}\Big]\exp(2sh/\chi),
\\
\nonumber
G_4 &= 2s \Big[\frac{1}{2 + \sqrt{2}}\Big]^2 \Big[\frac{1}{\chi}\Big]^4
\Big[\frac{2s+1}{4}\Big] \Big[\frac{s+1}{3}\Big]
\Big[\frac{2s+3}{8}\Big]\exp(2sh/\chi),
\end{align}
$D_3 = 4G_3$, and $D_4 = [4/(2 - \sqrt{2})] G_4$.
Then
\begin{lem}
\label{lem:8.1}
If $0 \le s \le 3/2$, $r \ge 2$, and $\chi:= 2 \amf_{\infty} + \gamma \ge 1$,
\begin{equation*}
\frac{G_{r+2}}{G_{r+1}} \le \Big[\frac{3}{4 \sqrt{2 + \sqrt{2}}}\Big] 
\frac{1}{\chi},
\qquad \text{so} \qquad
G_{r+2} \le  \Big[\frac{3}{4 \sqrt{2 + \sqrt{2}}}\Big(\frac{1}{\chi}\Big)
\Big]^{r-1} G_3.
\end{equation*}
\end{lem}
\begin{proof}
A calculation gives
\begin{equation*}
\frac{G_{r+2}}{G_{r+1}} = \Big[\frac{2s+r+1}{2r+4}\Big] u_{r+2} 
\Big[\frac{u_{r+2}}{u_{r+1}}\Big]^{r+1}.
\end{equation*}
Since $0 < u_{r+2} < u_{r+1}$, it follows that
\begin{equation*}
\frac{G_{r+2}}{G_{r+1}} \le \Big[\frac{2s+r+1}{2r+4}\Big]
\frac{1}{\chi}\Big[\frac{1}{2\cos(\pi/(2r+4))}\Big],
\end{equation*}
and if $r \ge 2$ and $s \le 3/2$,
\begin{equation*}
\frac{G_{r+2}}{G_{r+1}} \le \Big[\frac{3}{4\chi}\Big]
\Big[\frac{1}{2\cos(\pi/8)}\Big]
\le \Big[\frac{3}{4[\sqrt{2 + \sqrt{2}}]}\Big] \frac{1}{\chi}
\end{equation*}
and so
\begin{equation*}
G_{r+2} \le \Big[\frac{3}{4[\sqrt{2 + \sqrt{2}}]}
\Big(\frac{1}{\chi}\Big)\Big]^{r-1}G_3.
\end{equation*}
Furthermore, since we assume that $\chi \ge 1$, 
\begin{multline}
\label{G3bound}
G_3 = \Big[\frac{2s}{\chi}\Big] \Big[\frac{1}{\chi^2 3 \sqrt{3}}\Big]
 \Big[\frac{2s+1}{4} \Big]
\Big[\frac{s+1}{3}\Big]\exp(2sh/\chi)
\\
\le   \Big[\frac{2s}{\chi^3}\Big]  \Big[\frac{5}{18 \sqrt{3}}\Big]\exp(3h)
\le \sqrt{3}/2,
\end{multline}
for $h \le \ln(9/5)/3 \le 0.2$.
\end{proof}

\begin{lem}
\label{lem:8.2}
Assume $0 \le s \le 3/2$, $r \ge 2$, $\chi \ge 1$,
and $D_{r+1}$ is as in \eqref{8.9}.
Then
\begin{equation*}
D_{r+2} \le \Big[\frac{3}{4 \chi}\Big]^{r-1}D_3.
\end{equation*}
\end{lem}
\begin{proof}
A calculation gives
\begin{multline*}
\frac{D_{r+2}}{D_{r+1}} = \Big[\frac{u_{r+2}}{u_{r+1}}\Big]^{r+1}
u_{r+2} \Big[\frac{\sin(\pi/[2r+2])}{\sin(\pi/[2r+4])}\Big]^2 
\Big[\frac{2s+r+1}{2r+4}\Big]
\\
\le \frac{1}{\chi} \Big[\frac{1}{2 \cos(\pi/(2r+4))}\Big]
\Big[\frac{\sin(\pi/[2r+2])}{\sin(\pi/[2r+4])}\Big]^2 
\Big[\frac{r+4}{2r+4}\Big].
\end{multline*}
Using Taylor series expansions for $\sin(u)$, $u \ge 0$, we have
$\sin(u) \le u$ and $\sin(u) \ge u - u^3/6$. Hence,
\begin{equation*}
\frac{\sin(\pi/[2r+2]}{\sin(\pi/[2r+4]}
\le \Big[\frac{2r+4}{2r+2}\Big]
\Big[1 - \frac{1}{6}\Big(\frac{\pi}{2r+4}\Big)^2\Big]^{-1}.
\end{equation*}
Noting that for $r \ge 2$, the expression on the right hand side of
the above is a decreasing function of $r$, as are the other two functions
of $r$ in the bound for $D_{r+2}/D_{r+1}$, we see that an upper bound
for $D_{r+2}/D_{r+1}$ is obtained by setting $r=2$ in each of the expressions
above. This gives
\begin{equation*}
\frac{D_{r+2}}{D_{r+1}} \le \Big[\frac{1}{2\chi \cos(\pi/8)}\Big] 
\Big[\frac{4}{3}\Big]^2 
\Big[1 - \frac{1}{6}\Big(\frac{\pi}{8}\Big)^2\Big]^{-2}\Big[\frac{3}{4}\Big]
\le  \frac{3}{4 \chi},
\end{equation*}
and so 
\begin{equation*}
  D_{r+2} \le \Big[\frac{3}{4 \chi}\Big] D_{r+1} \le
\Big[\frac{3}{4 \chi}\Big]^{r-1}D_3.
\end{equation*}
\end{proof}
The following bound on $H_r$ is a direct consequence of the above estimates.
\begin{lem}
\label{lem:H-est}
Assume $0 \le s \le 3/2$, $r \ge 2$. Then
\begin{equation*}
H_{r+2} \le \mu \Big[\frac{\chi}{2s}\Big] \Big[\frac{3}{4 \chi}\Big]^{r-1} 4 G_3.
\end{equation*}
\end{lem}

\begin{lem}
\label{lem:8.4}
Suppose $0 \le s \le 3/2$, $r \ge 2$, $M_1$ is as in \eqref{6.12},
$M_2$ is as in Lemma~\ref{6.4}, and $H_{r+1} = \mu D_{r+1} (\chi/2s)$. Then
\begin{equation*}
M_1 = \frac{\mu D_{r+1} h^r}{1 - G_{r+1}^2 h^{2r+2}} + \frac{2s}{\chi},
\end{equation*}
and
\begin{multline}
\label{8.16}
M_2 =  M_1 + \Big[\frac{2s}{\chi}\Big]\frac{1}{1- [(G_{r+1}/H_{r+1})h]^2}
\\
=\frac{2s}{\chi}\Big[1 + \frac{H_{r+1} h^r}{1 - G_{r+1}^2 h^{2r+2}}
+ \Big(1- h^2 \frac{2s}{\mu \chi}\sin^2(\pi/[2r+2])\Big)^{-1}\Big].
\end{multline}
\end{lem}
\begin{proof}
Applying the definitions of $M_1$,
$M_2$, $D_{r+1}$, $G_{r+1}$, and $H_{r+1}$, we get
\begin{align*}
M_2 &= M_1 + \Big[\frac{\mu G_{r+1}}{H_{r+1}}\Big]
\Big[ \frac{1}{1- [(G_{r+1}/H_{r+1})h]^2}\Big]
\Big[\frac{1}{[\sin(\pi/[2r+2])]^2}\Big]
\\
&= M_1 + \Big[ \frac{2s}{\chi}\Big] \frac{1}{1- [(G_{r+1}/H_{r+1})h]^2}
\\
&=\frac{2s}{\chi} + \frac{\mu D_{r+1} h^r}{1 - G_{r+1}^2 h^{2r+2}}
+ \Big[\frac{2s}{\chi}\Big]\frac{1}{1- [(G_{r+1}/H_{r+1})h]^2}
\\
&= \frac{2s}{\chi} +  \Big[\frac{2s}{\chi}\Big]
\frac{H_{r+1} h^r}{1 - G_{r+1}^2 h^{2r+2}}
+ \Big[\frac{2s}{\chi}\Big]\frac{1}{1- (2s/[\mu\chi])[\sin(\pi/[2r+2])h]^2}
\\
&= \frac{2s}{\chi}\Big[1 + \frac{H_{r+1} h^r}{1 - G_{r+1}^2 h^{2r+2}}
+ \Big(1- h^2 \frac{2s}{\mu \chi}\sin^2(\pi/[2r+2])\Big)^{-1}\Big].
\end{align*}
\end{proof}
\begin{remark}
\label{rem:8.3}
Note that $D_{r+1}$ has  the factor $2s/\chi$, so the
identity \eqref{8.16} for $M_2$ does not blow up as $s \rightarrow 0$.
\end{remark}
\begin{remark}
\label{rem:8.4}
If we replace in $G_3$ and $D_3$, the quantity $\exp(2sh/\chi)$ by
$\exp(2s\amh_0/\chi)$, where $\amh_0$ is chosen so that \eqref{8.3} and \eqref{8.4}
are satisfied for $0< h \le \amh_0$, then we easily obtain a bound for
$M_2$ in the form
\begin{equation}
\label{8.16a}
M_2 \le \frac{2s}{\chi} \Big[2 + H_{r+1} h^r(1 + \tilde ch^{2r+2}) + ch^2\Big].
\end{equation}
where $\tilde c$ and $c$ are easily computable from
\eqref{8.16}.  Note that if $0 \le s \le 3/2$, $\chi = 2 \amf_{\infty}
+ \gamma \ge 1$, and $r \ge 2$, then $\chi \ge 1$ and
$\sin^2(\pi/[2r+2]) \le 1/4$. So, for $h \le 1/\sqrt{3}$,
\begin{equation*}
\Big(1- h^2 \frac{2s}{\mu \chi}\sin^2(\pi/[2r+2])\Big)^{-1}
\le (1 - 3h^2/4)^{-1} \le 1 + h^2,
\end{equation*}
i.e., we can take $c =1$. Similarly, for $h \le 0.2$, $G_{r+1}^2 \le 3/4$, so
we can also take $\tilde c =1$.
\end{remark}
Recall that we have to insure that $M > M_2$.
From \eqref{8.6a} we have that 
\begin{equation*}
M \ge \Big[\frac{4s}{\chi}\Big]\frac{1}{\kappa_2- \kappa_1}.
\end{equation*}
Comparing this expression to the bound for $M_2$ given in \eqref{8.16a}
and noting that $\kappa_2- \kappa_1 <1$, $M$ will be $> M_2$ if we
choose $h \le \amh_1$ sufficiently small so that it also satisfies
\begin{equation}
\label{h1cond}
2 + H_{r+1} h^r(1 + h^{2r+2}) + h^2 \le 2/(\kappa_2- \kappa_1).
\end{equation}

We can now state versions of Theorem~\ref{thm:6.5} and Theorem~\ref{thm:6.6}
in the context of this section.
\begin{thm}
\label{thm:8.1}
Assume that $r\ge 2$, $\chi \ge 1$, and $0 \le s \le 3/2$, and let
$\nu, \kappa_1$, and $\kappa_2$ be as described at the beginning of
this section.  Let $M_2$ and $H_{r+1}$ be as described in
Lemma~\ref{lem:8.4} and select $M$ such that \eqref{8.6} is satisfied.
Finally, assume that $h = \max_{i \in I} h_i \le \min(\amh_0,\amh_1,0.2)$. Then,
with $u = M\eta(r) h$, \eqref{8.3} and \eqref{8.4} are satisfied and
$M > M_2$.  Furthermore, (compare Theorem~\ref{thm:6.5})
$\bL_{s,\nu}(K(M;T)) \subset K(M';T)$, where $M' = \kappa_2 M < M$.
In addition, we have that for $H = H_{r+1}$,
\begin{equation*}
\lambda_s^{\nu}(1 - H h^r) \le \rmf(\bL_{s,\nu}) \le \lambda_s^{\nu}(1 + H h^r).
\end{equation*}
\end{thm}
\begin{proof}
The fact that $M_2 >M$ follows directly from the computations above.
Our selection of $r, \nu, h, M$ and $M'= \kappa_2 M$ shows that the inequality
in Theorem~\ref{3.1} is satisfied, so $\bL_{s,\nu}(K(M;T)) \subset
K(M';T)$. The inequality for $\rmf(\bL_{s,\nu})$ in Theorem~\ref{thm:8.1}
follows directly from Theorem~\ref{thm:6.5}.
\end{proof}
\begin{thm}
\label{thm:8.2}
Under the hypotheses of Theorem~\ref{thm:8.1}, we have
\begin{equation*}
[(1 + H h^r)^{-1} \rmf(\bL_{s,\nu})]^{1/\nu} \le \lambda_s \le
[(1 - H h^r)^{-1} \rmf(\bL_{s,\nu})]^{1/\nu},
\end{equation*}
where the entries of the matrices $[1 + H h^r]^{-1} \bL_{s,\nu}$ and
$[1 - H h^r]^{-1} \bL_{s,\nu}$ differ by $O(h^r)$.
\end{thm}

Using the inequalities of Theorem~\ref{thm:8.1}, we can obtain rigorous
upper and lower bounds on the Hausdorff dimension $s_*$ of the invariant
set associated with the transfer operator $L_s$ as follows.
Let $s_l$ and $s_u$ denote values of $s$ satisfying
\begin{equation*}
(1-Hh^r)^{-1}  \rmf(\bL_{s_u,\nu}) <1, \qquad
(1+Hh^r)^{-1}  \rmf(\bL_{s_l,\nu}) >1.
\end{equation*}
It follows immediately from Theorem~\ref{thm:8.1} that $\lambda_{s_u}^{\nu} <
1$ and $\lambda_{s_l}^{\nu} >1$.  Since the spectral radius $\lambda_s$ of
$L_s$ is a decreasing function of $s$, there will be a value $s_*$ satisfying
$s_l < s_* < s_u$ for which $\lambda_{s_*}^{\nu} = 1$, or equivalently
$\lambda_{s_*} = 1$.  The value $s_*$ gives the Hausdorff dimension $s_*$ of
the invariant set associated with the transfer operator $L_s$.

\section{Numerical Computations}
\label{sec:num-comp}
In this final section, we present results of computations of the
Hausdorff dimension $s$ for various choices of sets of continued
fractions, maximum mesh size $h$, piecewise polynomial degree $r$, and
number of iterations $\nu$ of the map, where $\nu=1$ corresponds to
the original map.  These computations include choices of the above
parameters (especially the number of iterations $\nu$), for which the
hypotheses of our theorems are satisfied, but also computations which
obtain the same results when the mappings are not iterated (denoted by
$\nu =1^*$ in Table~\ref{tb:t2} below). We note the obvious fact that
as the number of iterations increase, the complexity of the operator
increases, so the time to compute the maximum eigenvalue of the
resulting matrix operator increases as well. To keep the time involved
in the various computations to be reasonable, we have elected to
compute to fewer digits those sets with a large value of $\nu$,
especially if the set $E$ contains many digits.  We have also done
additional computations, not reported in Table~\ref{tb:t2}, which
indicate that the method is robust with respect to the choices of $h$,
$r$, and $\nu$.

\begin{table}[htp]
\footnotesize
\caption{Computation of Hausdorff dimension $s$ for various choices of 
sets of continued fractions, maximum mesh size $h$, piecewise polynomial
degree $r$, and number of iterations $\nu$.}
\label{tb:t2}
\begin{center}
\begin{tabular}{|c|c|c|c|c|}
\hline
Set E   &   $r$  & $h$  & $\nu$ &\\
\hline 
& \multicolumn{4}{l|}{$s =$}\\
\hline \hline 
E[1,2] & 14 & 0.0002  & 7 & \\
\hline 
& \multicolumn{4}{l|}
{$s =$ 0.531 280 506 277 205 141 624 468 647 368 471 785 493 059 109 018 398}
\\
\hline\hline  
E[1,3] & 8  & 5.0e-05  & 6   \\
\hline 
& \multicolumn{4}{l|}{$s =$
0.454 489 077 661 828 743 845 777 611 651}\\
\hline\hline 
E[1,4] & 8  & 5.0e-05  & 6   \\
\hline 
& \multicolumn{4}{l|}{$s =$
0.411 182 724 774 791 776 844 805 904 696}\\
\hline\hline 
E[2,3] & 8  & 5.0e-05  & 3 & \\
\hline 
& \multicolumn{4}{l|}
{$s =$ 0.337 436 780 806 063 636 304 494 910 387}
\\
\hline\hline  
E[2,4] & 8  & 5.0e-05  & 3 & \\
\hline 
& \multicolumn{4}{l|}
{$s =$ 0.306 312 768 052 784 030 277 908 307 445}
\\
\hline\hline
E[3,4] & 8  & 5.0e-05 & 3 & \\
\hline 
& \multicolumn{4}{l|}
{$s =$ 0.263 737 482 897 426 558 759 863 384 275}
\\
\hline\hline
E[3,7] & 18   & .01  & 3 & \\
\hline 
& \multicolumn{4}{l|}
{$s =$ 0.224 923 947 191 778 989 184 480 593 490}
\\
\hline\hline
E[10,11]  & 20  & .002  & 2 & \\
\hline 
& \multicolumn{4}{l|}
{$s =$ 0.146 921 235 390 783 463 311 108 628 515 904 073 067 083 129 676 755}
\\
\hline \hline 
E[100, 10,000] & 20  & .002   & 1 & \\
\hline 
&  \multicolumn{4}{l|}
{$s =$ 0.052 246 592 638 658 878 652 588 416 300 508 181 012 676 284 431 681}
\\
\hline \hline 
E[1,2,3] & 5  & .0001  & 5 & \\
\hline 
&  \multicolumn{4}{l|}
{$s =$ 0.705 660 908 028 738}
\\
\hline \hline 
E[1,3,4] & 5  & .0001   & 5 & \\
\hline 
&  \multicolumn{4}{l|}
{$s =$ 0.604 242 257 756 515 }
\\
\hline \hline 
E[1,3,5] & 8  & .001   & 6 & \\
\hline 
&  \multicolumn{4}{l|}
{$s =$ 0.581 366 821 182 975  }
\\
\hline \hline
E[1,4,7] & 6  & .001   & 6 & \\
\hline 
&  \multicolumn{4}{l|}
{$s =$ 0.517 883 757 006 911}
\\
\hline \hline
E[2,3,4] & 16  & .005  & 4 & \\
\hline 
&  \multicolumn{4}{l|}
{$s =$ 0.480 696 222 317 573 041 322 515 564 711}
\\
\hline \hline 
E[1,2,3,4] & 8  & .005  & 6 & \\
\hline 
&  \multicolumn{4}{l|}
{$s =$ 0.788 945 557 483 153}
\\
\hline \hline 
E[2,3,4,5] & 16  & .005  & 4 & \\
\hline 
&  \multicolumn{4}{l|}
{$s =$ 0.559 636 450 164 776 713 312 144 913 530}
\\
\hline \hline 
E[1,2,3,4,5] & 5   & .0005   & 5 & \\
\hline 
& \multicolumn{4}{l|}{$s =$ 0.836 829 443 681 209}\\
\hline \hline 
E[2,4,6,8,10] & 7  & .005  & 3 & \\
\hline 
& \multicolumn{4}{l|}
{$s =$ 0.517 357 030 937 017}
\\
\hline\hline 
E[1,\ldots,10] & 10  & .01  & 1* &\\
\hline 
& \multicolumn{4}{l|}{$s =$ 0.925 737 591 146 765}\\
\hline \hline
E[1,\ldots,34] & 10  & .01  & 1* &\\
\hline 
& \multicolumn{4}{l|}{$s =$ 0.980 419 625 226 980 }\\
\hline \hline
E[1,3,5,\ldots,33] & 10  & .01  & 1* &\\
\hline 
& \multicolumn{4}{l|}{$s =$ 0.770 516 008 717 163}\\
\hline \hline
E[2,4,6,\ldots,34] & 10  & .01  & 1* &\\
\hline 
& \multicolumn{4}{l|}{$s =$ 0.633 471 970 241 089}\\
\hline 
\end{tabular}
\end{center}
\end{table}

In addition to the results presented in Table~\ref{tb:t2}, we have also used
our method to compute the Hausdorff dimension of the set $E[1,2]$ with
degree $r =36$, $h = .001$, and $\nu = 1$ using multiple precision
with 108 digits.  Although this choice does not satisfy the hypotheses
of our theorem, the result agrees to 100 decimal places with the
result in \cite{JP100}.  While we do not have a proof that our method
also works in the non-iterated situation, we do not have any examples
where it fails.  We conjecture that the limitation is the method of
proof and not the underlying method.

We next discuss how to control the size of some of the constants that appear
in the error estimates.  We begin with the constant $\mu = \max_i h_i/\min_i
h_i$.  Recall that to satisfy hypothesis (H3), for a given positive integer
$\nu$, we need to determine pairwise disjoint, nonempty compact intervals
$[a_i,b_i] \subset S$, $1 \le i \le I$, such that for every $\omega \in
\Omega_{\nu}$, there exists $i =i(\omega)$, $1 \le i \le I$, such that
$\theta_{\omega}(S) \subset [a_{i},b_{i}]$.  By the results in the previous
section, we can take $S = [\amf_{\infty}, \bmf_{\infty}]$ and we then order
the $\omega \in \Omega_{\nu}$, such that the sets $\theta_{\omega}(S)$ are
ordered with $a_i < a_{i+1}$.  Although we could use the domain consisting
of the union of the sets  $\theta_{\omega}(S)$, this can lead to very
small subinterval sizes. Instead, we determine a new domain 
by iterating only $\nu^{\prime}$ times, while still using 
the mappings obtained by $\nu$ iterations to calculate the mapping $L$.
The constant $\nu^{\prime}$ is determined so that the length of the smallest
interval will be $\ge h_{max}/\mu.$

Although we have not done the computations using interval arithmetic, 
we have only included the number of digits in each computation that we
expect to be correct, which is always less than the number of digits
provided by {\it Matlab} for the precision we have specified. For
computations that use more digits than provided by standard {\it Matlab} 
computations, we have used the {\it Advanpix} multiprecision toolbox.

Because the theory developed in the previous sections involves the
computation and estimates for many constants and parameters, we next
provide, for the benefit of the reader, details for the specific
example of $E[1,4,7]$, corresponding to the computation for $r=6$, $h=0.001$,
$\nu = 6$, and $s =0.518$, shown in Table~\ref{tb:t2}.  The computational
domain is determined by only iterating $\nu^{\prime} = 2$ times and the total
number of subintervals used is $= 191$. With these choices,
$\gamma = 1$ and $\Gamma =7$. Then by \eqref{5.1},
\begin{equation*}
\amf_{\infty} = - \frac{\gamma}{2} + \sqrt{(\gamma/2)^2 + (\gamma/\Gamma)}
\approx 0.127 \quad \text{and} \quad
\bmf_{\infty} = (\Gamma/\gamma) \amf_{\infty} \approx 0.8875.
\end{equation*}
Working on the interval $[\amf_{\infty}, \bmf_{\infty}]$, we have
$\chi = \amf_{\infty} + \bmf_{\infty}^{-1} = 2 \amf_{\infty} + \gamma 
\approx 1.25$.  From
\eqref{5.5}-- \eqref{5.7}, we have for $\nu = 6$, that
\begin{equation*}
c(\nu) = [\tilde B_{\nu-1} \amf_{\infty} + \tilde B_{\nu}]^{-2} = 0.0051
\end{equation*}
From \eqref{3.1}, we have that $\eta(r) = 1/2$ and from Lemma~\ref{lem:3.2}
that
\begin{equation*}
\psi(r) = (2/\pi) \ln(r +1) + 3/4 \approx 2.0
\end{equation*}
From \eqref{M0nuchoice}, since we are working on the interval 
$[\amf_{\infty}, \bmf_{\infty}]$, we take
\begin{equation*}
M_0(\nu) = 2/(\amf_{\infty} + \bmf_{\nu}^{-1}) = 2/(\amf_{\infty} + \amf_{\nu-1}
+ \gamma) = 1.5954
\end{equation*}
Setting
\begin{equation*}
\kappa_1 = c(\nu) 2 \eta(r) r^2 \psi(r) =  0.364   <1,
\end{equation*}
we choose $\kappa_2 = (1 + \kappa_1)/2 = 0.6823$, so that $\kappa_1 <
\kappa_2 <1$.  Next, we choose
\begin{equation*}
M = \frac{4s}{\amf_{\infty} + \amf_{\nu-1} + \gamma}\frac{1}{\kappa_2 - \kappa_1}
= 5.2022
\end{equation*}
and $M^{\prime} = \kappa_2 M$.
One of the conditions on the mesh size $h = \max_{i \in I} h_i$ is
that it be sufficiently small so that \eqref{8.3} and \eqref{8.4} are
satisfied.  For our example, setting $u = M \eta(r) h$, we have
\begin{gather*}
\psi(r) u \exp(u) = .0052 < 1,
\\
\frac{\kappa_1 \exp(u)}{1- \psi(r) u \exp(u)} = 0.3674
< 0.5234 = \kappa_2 - \frac{s   M_0(\nu)}{M}.
\end{gather*}

The second condition, coming from \eqref{h1cond} is that
\begin{equation}
\label{secondcond}
2 + H_{r+1}(h^r + h^{2r+2}) + h^2 \le 2/(\kappa_2 - \kappa_1),
\end{equation}
where, combining \eqref{8.9}, \eqref{Hr1def}, and Lemma~\ref{lem:8.2}, we have
\begin{equation*}
H_{r+1} \le \mu  \frac{\chi}{2s} D_{r+1}
\le \mu  \frac{\chi}{2s} \Big[\frac{3}{4 \chi}\Big]^{r-2} D_3
\le \mu  \frac{\chi}{2s} \Big[\frac{3}{4 \chi}\Big]^{r-2} 4 G_3,
\end{equation*}
where
\begin{equation*}
G_3 = 2s \Big[\frac{1}{\chi\sqrt{3}}\Big]^3\Big[\frac{2s+1}{4} \Big]
\Big[\frac{s+1}{3}\Big]\exp(2sh/\chi).
\end{equation*}
Combining these results and setting $\mu$, the ratio of the maximum 
subinterval size to the minimum subinterval size,  $= 3.76$, the value
calculated by the computer code, we get
\begin{equation*}
H_{r+1} \le 0.0608.
\end{equation*}
Since $2/(\kappa_2 - \kappa_1) \ge 6$, it is clear that
\eqref{secondcond} is satisfied.

We thus have satisfied the conditions of Theorems~\ref{thm:8.1} and
\ref{thm:8.2}, which guarantee the rigorous bounds we use to obtain
rigorous upper and lower bounds on the Hausdorff dimension of 
the set $E[1,4,7]$.

\bibliographystyle{amsplain}

\end{document}